\documentclass[a4paper,10pt,reqno]{amsart}

\usepackage{a4wide}
\usepackage[english]{babel}
\usepackage[latin1]{inputenc}
\usepackage{amsmath,amssymb,amsfonts,amsthm,amscd}
\usepackage[mathscr]{eucal}
\usepackage{enumitem}
\usepackage{latexsym}
\usepackage{multicol}
\usepackage{color}
\usepackage{multirow}

\newcommand{\N}{\mathbb{N}}

\newcommand{\R}{\mathbb{R}}
\newcommand{\C}{\mathbb{C}}
\renewcommand{\H}{\mathbb{H}}
\newcommand{\df}{\,\mathrm{d}}
\newcommand{\M}{\mathbb{M}}
\newcommand{\E}{\mathbb{E}}
\newcommand{\LL}{\mathbb{L}}

\newcommand{\arcsinh}{\mathop{\rm arcsinh}\nolimits}
\newcommand{\vol}{\mathop{\rm vol}\nolimits}
\newcommand{\area}{\mathop{\rm area}\nolimits}
\newcommand{\Length}{\mathop{\rm length}\nolimits}

\newcommand{\Nil}{\mathrm{Nil}_3}
\newcommand{\PSL}{\widetilde{\mathrm{SL}}_2(\mathbb{R})}
\renewcommand{\div}{\mathop{\rm div}\nolimits}

\newtheorem{theorem}{Theorem}
\newtheorem*{theorem*}{Theorem}
\newtheorem{proposition}{Proposition}
\newtheorem{corollary}{Corollary}
\newtheorem{lemma}{Lemma}

\theoremstyle{definition}

\theoremstyle{remark}
  \newtheorem{remark}{Remark}

\numberwithin{equation}{section}
\setlist[itemize]{parsep=0.4em}
\setlist[enumerate]{parsep=0.4em}
\title{Height and area estimates for constant mean curvature graphs in $\E(\kappa,\tau)$-spaces}
\date{}

\author{Jos\'{e} M. Manzano}
\address{Dipartimento di Scienze Matematiche\\
Politecnico di Torino \\ Corso Duca degli Abruzzi, 24 -- 10129 Torino (Italy)}
\email{manzanoprego@gmail.com}

\author{Barbara Nelli}
\address{Dipartimento di Ingegneria e Scienze dell'Informazione e Matematica\\
Universit\`{a} dell'Aquila\\
Via Vetoio Loc. Coppito -- 67100 L'Aquila (Italy)}
\email{nelli@univaq.it}

\thanks{This research was partially supported by PRIN-2010NNBZ78-009. The first author was also partially supported by Spanish MCyT-Feder research project MTM2011-22547. The second author would like to thank Massimiliano Pontecorvo for the kind hospitality during the preparation of this work.}



\begin{document}

\begin{abstract}
 We obtain area growth estimates for constant mean curvature graphs in $\E(\kappa,\tau)$-spaces with $\kappa\leq 0$, by finding sharp upper bounds for the volume of geodesic balls in $\E(\kappa,\tau)$. We focus on complete graphs and graphs with zero boundary values. For instance, we prove that entire graphs in $\E(\kappa,\tau)$ with critical mean curvature have at most cubic intrinsic area growth. We also obtain sharp upper bounds for the extrinsic area growth of graphs with zero boundary values, and study distinguished examples in detail such as invariant surfaces, $k$-noids and ideal Scherk graphs. Finally we give a relation between height and area growth of minimal graphs in the Heisenberg space ($\kappa=0$), and prove a Collin-Krust type estimate for such minimal graphs.
\end{abstract}

\maketitle

\section{Introduction}

Constant mean curvature surfaces in simply-connected homogeneous $3$-manifolds have been object of study of many authors in the last decade. Special attention has been given to those $3$-manifolds with isometry group of dimension at least $4$, which are classified in a $2$-parameter family $\E(\kappa,\tau)$, $\kappa,\tau\in\R$, with the exception of the hyperbolic space $\H^3$. Also $\E(\kappa,\tau)$ admits a Riemannian submersion with bundle curvature $\tau$ over $\M^2(\kappa)$, the simply-connected $2$-dimensional manifold with constant curvature $\kappa$, such that the fibers of the submersion are the integral curves of a distinguished unit Killing vector field~\cite{Dan,Man}. It arises the natural question of studying graphs (i.e., sections of the Riemannian submersion) with constant mean curvature, over domains of $\M^2(\kappa)$, as a non-parametric version of the constant mean curvature condition (see Section~\ref{minimal-graph-section}).

 A fundamental tool in the comprehension of surfaces in $\E(\kappa,\tau)$-spaces is Daniel correspondence~\cite{Dan}, which couples isometric constant mean curvature surfaces in different $\E(\kappa,\tau)$-spaces with different constant mean curvatures, and respects locally the graphical condition.

We will actually focus on the case of constant mean curvature graphs in $\E(\kappa,\tau)$, for any $\tau\geq 0$ and $\kappa\leq 0$. If $\tau\neq 0$, this restriction leads to the Heisenberg group $\Nil(\tau)=\E(0,\tau)$ and $\PSL$ for $\kappa<0$. If $\tau=0$, one has $\E(0,0)=\R^3$ and $\E(\kappa,0)=\H^2(\kappa)\times\R$ for $\kappa<0$. We are skipping the case $\kappa>0$, i.e., when the submersion is over the round sphere $\mathbb{S}^2(\kappa)$, since the results we are looking for do not make sense in that case.

Our first aim is to evaluate how fast the area of a minimal graph in $\E(\kappa,\tau)$ can grow. Up to our knowledge, questions related to area growth of surfaces in $\E(\kappa,\tau)$ have not been tackled yet. 

Here we will propose three different notions of area growth of a surface $\Sigma\subset\E(\kappa,\tau)$ as the growth of the function $R\mapsto\area(\Sigma\cap A_R)$, where $A_R$ is either the geodesic ball in $\Sigma$ of radius $R$ (intrinsic area growth), or the extrinsic 
geodesic ball in $\E(\kappa,\tau)$ of radius $R$ (extrinsic area growth), or a solid cylinder of radius $R$ (i.e., the preimage of a disk of radius $R$ in $\M^2(\kappa)$ by the submersion) in $\E(\kappa,\tau)$ (cylindrical area growth). These three notions are independent of the point where the sets $A_R$ are centered, and it is easy to see that intrinsic area growth is always slower than the extrinsic one, which is in turn slower than the cylindrical one.

Our first estimate on the area will rely on a detailed study of the volume of geodesic balls in $\E(\kappa,\tau)$. For instance, we obtain that geodesic balls $B_R$ of radius $R$ in $\Nil(\tau)$ have quartic area growth, in the sense that $R^{-4}\vol(B_R)$ remains bounded between two positive constants, when $R$ is bounded away from zero (Proposition~\ref{volume-geod-nil}), in contrast to the case of $\R^3$, where this growth is cubic. Moreover, we give explicit expressions for the geodesics of $\E(\kappa,\tau)$, $\kappa\leq 0$, in terms of initial conditions.

The key idea in the extrinsic estimates of the area is to get a relation between the area of the intersection of the  surface with an extrinsic ball $B_R$ and geometric quantities computed on the base $\M^2(\kappa)$  (Lemmas~\ref{lemma:area-estimate} and~\ref{lemma:area-estimate2}). The main result (Theorem~\ref{thm:area-growth-minimal}) states that, if $\Sigma\subset\E(\kappa,\tau)$ is a minimal graph over a domain $\Omega\subset\M^2(\kappa)$ such that either $\Sigma$ extends to $\partial\Omega$ with zero boundary values, or $\Length(\partial(\Omega\cap D_R))$ is suitably controlled, being $D_R=\pi(B_R)$ a disk of radius $R$ in $\M^2(\kappa)$, then:
\begin{enumerate}[label=(\alph*)]
	\item If $\E(\kappa,\tau)=\R^3$, then $\Sigma$ has at most quadratic extrinsic area growth.
	\item If $\E(\kappa,\tau)=\Nil(\tau)$, then $\Sigma$ has at most cubic extrinsic area growth.
 	\item If $\kappa<0$, then $\Sigma$ has at most extrinsic area growth of order $R\mapsto R\,e^{\sqrt{-\kappa}R}$.
\end{enumerate}

As a first consequence of the extrinsic estimate, we are able to analyze the intrinsic area growth of a complete  graph $\Sigma$ with constant mean curvature $H$ in $\E(\kappa,\tau)$ (Theorem~\ref{thm:area-graphs}). If $4H^2+\kappa>0$, then  it is proved in~\cite{MPR} that $\kappa>0$ and $\Sigma=\mathbb{S}^2(\kappa)\times\{t_0\}$ in $\mathbb{S}^2(\kappa)\times\R$. If $4H^2+\kappa=0$ (i.e., $\Sigma$ has critical mean curvature), then  $\Sigma$ has at most cubic intrinsic area growth. Finally, if the mean curvature is subcritical ($4H^2+\kappa<0$), then the intrinsic area grows at most exponentially as $Re^{R\sqrt{-\kappa-4H^2}}$. The same estimates hold for extrinsic area growth if $H=0$.

As a second application, we get some intrinsic and extrinsic properties of important examples in the theory, such as horizontal umbrellas, minimal graphs in $\Nil(\tau)$ invariant by a one parameter family of ambient isometries (classified by Figueroa-Mercuri-Pedrosa~\cite{FMP}), symmetric $k$-noids with subcritical constant mean curvature in $\H^2(\kappa)\times\R$~\cite{MorRod,Plehnert,Pyo}, and ideal Scherk graphs (i.e., graphs on unbounded domains of $\H^2(\kappa)$ bounded by ideal polygons with a finite number of sides, taking $\pm\infty$-values alternately along the boundar) with subcritical constant mean curvature~\cite{CR,FM,Me}. In the case of $k$-noids and ideal Scherk graphs, we conclude that they have intrinsic quadratic area growth, so they are parabolic (i.e., the only non-positive subharmonic functions on the surface are the constant ones) by a classical result of Cheng and Yau~\cite{CY}.

As a byproduct of our technique we obtain some intermediate results of interest by themselves. On the one hand, we get that the area of the projection to $\M^2(\kappa)$, $k\leq 0$, of a complete graph with constant mean curvature $H$ is finite if and only if it is an ideal Scherk graph (Proposition~\ref{prop:finite-area}), which allows us to prove that the class of ideal Scherk graphs is preserved by the Daniel correspondence (Corollary~\ref{coro:scherk-preserved}). On the other hand, we discover that Figueroa-Mercuri-Pedrosa examples are parabolic though their intrinsic area growth is exactly cubic (Proposition \ref{FMP-cubic}). We employ this to correct a small mistake in the Bernstein-type theorem for parabolic horizontal graphs given by the first author, P\'{e}rez and Rodr\'{i}guez~\cite[Theorem 3]{MPR}: we deduce that a complete parabolic minimal surface in $\Nil(\tau)$ which is transversal to a non-vertical right-invariant Killing vector field is either a plane or congruent to an invariant surface (Theorem~\ref{thm:correction}).

Lower bounds on the area come from analyzing the cylindrical area growth (see 
Section~\ref{sec:cylindrical-area}), which is a suitable tool to study the area of entire graphs. We obtain that an entire graph in $\Nil(\tau)$ has at least cubic cylindrical area growth, while in $\PSL$ it has at least cylindrical area growth of order $e^{R\sqrt{-\kappa}}$ (Corollary \ref{coro:cylindrical-growth}). We emphasize that these estimates do not have assumptions on the mean curvature of the graph, and follow from a classical application of the divergence theorem.

It is worthwhile noticing that the cylindrical area growth of entire minimal graphs in $\Nil(\tau)$ is at least cubic, and the extrinsic one is at most cubic. Hence proving that they coincide for some entire minimal graph $\Sigma$, would ensure that $\Sigma$ has exactly cubic extrinsic area growth. Here height estimates come in handy and tell us that the slower the height of the surface grows, the better the extrinsic area growth is controlled. In particular, if the height of an entire minimal graph $\Sigma\subset\Nil(\tau)$ grows at most quadratically with respect to the distance to the origin in the base $\R^2$, then $\Sigma$ has exactly cubic extrinsic area growth (Corollary~\ref{coro:height-area-estimate-nil}). Height is always measured with respect to the usual zero section in $\Nil(\tau)$, so this situation applies to many known explicit examples of entire minimal graphs (e.g., see~\cite{C,Dan2,FMP,NST}). In this sense, the last part of the paper deals with height estimates for minimal graphs in $\Nil(\tau)$.

On the one hand, we obtain that the height of an entire minimal graph $\Sigma\subset\Nil(\tau)$ grows at most cubically, which ensures that the extrinsic area growth of $\Sigma$ is between quadratic and cubic (Theorem~\ref{last-area-estimate}). This is achieved by getting a global gradient estimate for entire spacelike graphs in the Lorentz-Minkowski space $\LL^3$ with positive constant mean curvature (Lemma~\ref{lemma:gradient-estimate-L3}), based on the work of Cheng and Yau~\cite{CY2} and Treibergs~\cite{Treibergs} through the Calabi-type correspondence by Lee~\cite{Lee}. It is worth emphasizing that our gradient estimate for $\Sigma$ is sharper than the general estimates for the angle function in Killing submersions given by Rosenberg, Souam and Toubiana in~\cite{RST}. As a consequence, we improve a result of Espinar~\cite[Corollary 5.2]{Espinar} by showing that a complete orientable stable surface with constant mean curvature $H$ in $\E(\kappa,\tau)$, with $\tau\neq 0$ and $4H^2+\kappa\geq0$, whose angle function is square-integrable must be a vertical cylinder (Corollary~\ref{coro:estability}).

On the other hand, we complete the study of the height by getting a sharp Collin-Krust type estimate~\cite{CK}, which establishes that the height of a minimal graph in $\Nil(\tau)$ with zero boundary values over an unbounded domain cannot grow less than linearly (Theorem \ref{collin-krust-theo} and Corollary~\ref{coro:sublineal-height}). Essentially we prove that all such graphs grow at least as catenoids (i.e., the situation is similar to $\R^3$, where Collin and Krust proved at least logarithmic height growth~\cite{CK}; see also the generalization by Leandro and Rosenberg~\cite{LR}).

As a final remark we point out that our results also yield a new estimate for minimal surfaces in $\R^3$ (Theorem \ref{thm:area-growth-minimal}) and that some of them can be directly generalized to the setting of Killing submersions (Remark \ref{rmk:generalization}). 

The paper is organized as follows. In the second section we deal with geodesic balls in $\E(\kappa,\tau)$ and describe their shape and volume. In the third section we recall the equation of a minimal graph in $\E(\kappa, \tau)$, and describe many known examples. The fourth and fifth sections contain the most important results of the paper about the intrinsic, extrinsic and cylindrical area growth of constant mean curvature surfaces in $\E(\kappa,\tau)$-spaces. In the sixth section we obtain the height estimates \`{a} la  Collin-Krust. Table \ref{table-examples} summarizes our principal results about area growth.

\begin{table}[!h]
\centering
\begin{tabular}{|c|c|c||c|c|c|c|}  
\hline
Surface & Curvature & Space & EAG & CAG & IAG & CT\\
\hline\hline
\multirow{2}{*}{Umbrellas}& \multirow{2}{*}{$H=0$} &$\Nil(\tau)$  &  $R^3$&$R^3$ &$R^3$ & \multirow{2}{*}{Hyp.}\\  
\cline{3-6}
&&$\kappa<0$  & $e^{R\sqrt{-\kappa}}$ &$e^{R\sqrt{-\kappa}}$ & $e^{R\sqrt{-\kappa}}$&\\  
\hline
FMP surfaces& $H=0$& $\Nil(\tau)$ &$R^3$  & $R^3$ & $R^3$ &\multirow{3}{*}{Par.}\\
\cline{1-6}
Ideal Scherk& \multirow{2}{*}{$4H^2+\kappa<0$} & $\E(\kappa, \tau)$&$\leq R^2$& & \multirow{2}{*}{$\leq R^2$} &\\
\cline{1-1}\cline{3-3}
$k$-noids& & $\H^2(\kappa)\times\R$& ($H=0$)&  & &\\
\hline
\multirow{4}{*}{Entire graphs}& $H=0$ &$\Nil(\tau)$ & $\geq R^2, \leq R^3$  &$\geq R^3, \leq R^4$ & $\leq R^3$ &\\
\cline{2-7}
 & $4H^2+\kappa=0$&\multirow{3}{*}{$\kappa<0$} & &$\geq e^{R\sqrt{-\kappa}}$  & $\leq R^3$ & \\
\cline{2-2}\cline{4-7}
 & $4H^2+\kappa<0$& & & $\geq e^{R\sqrt{-\kappa}}$& $\leq Re^{R\sqrt{-\kappa-4H^2}}$& \\
 \cline{2-2}\cline{4-7}
 & $H=0$ & & $\leq Re^{R\sqrt{-\kappa}}$  & $\geq e^{R\sqrt{-\kappa}}$  & $\leq Re^{R\sqrt{-\kappa}}$ & \\
\hline
Graphs with & \multirow{3}{*}{$H=0$} & $\R^3$ &  $\leq R^2$&  & $\leq R^2$ &\\
\cline{3-7}
zero boundary & &$\Nil(\tau)$ &  $\leq R^3$&  & $\leq R^3$ &\\
\cline{3-7}
values & &$\kappa<0$ &  $\leq Re^{R\sqrt{-\kappa}}$&  & $\leq Re^{R\sqrt{-\kappa}}$ &\\
\hline
\end{tabular}
\medskip
\caption{EAG=Extrinsic Area Growth, CAG=Cylindrical Area Growth,  IAG= Intrinsic Area Growth, CT= Conformal Type (hyperbolic or parabolic).}
\label{table-examples}
\end{table}

\section{Geodesics in $\E(\kappa,\tau)$-spaces.}

Given $\kappa,\tau\in\R$, we will consider the model for the $3$-manifold $\E(\kappa,\tau)$ as
\[\mathbb{E}(\kappa,\tau)=\left\{(x,y,z)\in\R^3:1+\tfrac{\kappa}{4}(x^2+y^2)>0\right\},\]
endowed with the only Riemannian metric such that
\begin{align*}
 E_1&=\frac{\partial_x}{\lambda}-\tau y\,\partial_z,& 
 E_2&=\frac{\partial_y}{\lambda}+\tau x\,\partial_z,&
 E_3&=\partial_z,
\end{align*}
defines a global orthonormal frame, where
\[\lambda(x,y,z)=\left(1+\frac{\kappa}{4}(x^2+y^2)\right)^{-1},\qquad (x,y,z)\in\E(\kappa,\tau).\]
The projection to the first two components $(x,y,z)\mapsto(x,y)$, is a Riemannian submersion with bundle curvature $\tau$ onto $\mathbb{M}^2(\kappa)$, the simply-connected surface with constant curvature $\kappa$. The fibers of the submersion are geodesics, and coincide with the integral curves of the unit Killing vector field $E_3$. The Levi-Civita connection $\overline\nabla$ on $\E(\kappa,\tau)$ in the frame $\{E_1,E_2,E_3\}$ is given by 
\begin{equation}\label{eqn:levi-civita}
\begin{array}{lclcl}
\overline\nabla_{E_1}E_1=\frac{\kappa}{2}yE_2,&&\overline\nabla_{E_1}E_2=-\frac{\kappa}{2}E_1+\tau E_3,&&\overline\nabla_{E_1}E_3=-\tau E_2,\\
\overline\nabla_{E_2}E_1=-\frac{\kappa}{2}xE_2-\tau E_3,&&\overline\nabla_{E_2}E_2=\frac{\kappa}{2}xE_1,&&\overline\nabla_{E_2}E_3=\tau E_1,\\
\overline\nabla_{E_3}E_1=-\tau E_2,&&\overline\nabla_{E_3}E_2=\tau E_1,&&\overline\nabla_{E_3}E_3=0.
\end{array}
\end{equation}

Let us describe the equations of the geodesics in $\E(\kappa, \tau)$.  Given a curve $\gamma:\R\to\E(\kappa,\tau)$, it can be expressed as $\gamma(t)=(x(t),y(t),z(t))\in\R^3$ so $\gamma'(t)=x'(t)\partial_x+y'(t)\partial_y+z'(t)\partial_z=\sum_{k=1}^3a_k(t)E_k$, for some functions $a_k:\R\to\R$. It is straightforward to check that
\begin{align*}
a_1&=\frac{x'}{1+\frac{\kappa}{4}(x^2+y^2)},&
a_2&=\frac{y'}{1+\frac{\kappa}{4}(x^2+y^2)},&
a_3&=z'+\tau\frac{yx'-xy'}{1+\frac{\kappa}{4}(x^2+y^2)}.
\end{align*}
By means of the Levi-Civita connection~\eqref{eqn:levi-civita}, the condition $\overline\nabla_{\gamma'}{\gamma'}=0$ is easily developed. We conclude  that $\gamma$ is a geodesic if and only if $(a_1,a_2,a_3)$ is a solution to the following \textsc{ode} system:
\begin{equation}\label{eqn:ode-geodesics}
\left\{\begin{array}{l}
a_1'=-\frac{\kappa}{2}xa_2^2+\frac{\kappa}{2}ya_1a_2-2\tau a_2a_3,\\
a_2'=-\frac{\kappa}{2}ya_1^2+\frac{\kappa}{2}xa_1a_2+2\tau a_1a_3,\\
a_3'=0.
\end{array}\right.
\end{equation}
On the other hand, we know that if $\gamma$ is a geodesic, then $\pi\circ\gamma:\R\to\mathbb M^2(\kappa)$ has constant geodesic curvature and constant speed, and the angle function $a_3=\langle\gamma',E_3\rangle$ is also constant~\cite{Man}, which allows us to obtain the explicit solutions of~\eqref{eqn:ode-geodesics} given below.

In the sequel we will denote by $B_R(p)$ (resp. $D_R(x)$) the geodesic ball of $\E(\kappa,\tau)$ (resp. $\M^2(\kappa)$) of radius $R\geq 0$ centered at $p\in\E(\kappa,\tau)$ (resp. $x\in\M^2(\kappa)$).

\subsection{Geodesics balls in $\Nil(\tau)$}
Given $\phi\in[0,\pi]$, $\phi\neq\frac{\pi}{2}$, and $\theta\in\R$, it is straightforward to check that
\begin{equation}\label{eqn:geodesics}
\begin{aligned}
 x(t)&=\frac{\tan(\phi)}{2\tau}\left(\cos(2\tau\cos(\phi)t+\theta)-\cos(\theta)\right),\\
 y(t)&=\frac{\tan(\phi)}{2\tau}\left(\sin(2\tau\cos(\phi)t+\theta)-\sin(\theta)\right),\\
 z(t)&=\frac{1+\cos^2(\phi)}{2\cos(\phi)}\,t-\frac{\tan^2(\phi)}{4\tau}\sin(2\tau\cos(\phi)t),
\end{aligned}
\end{equation}
defines a complete geodesic in $\Nil(\tau)$ such that $x(0)=y(0)=z(0)=0$, and $x'(0)=-\sin(\theta)\sin(\phi)$, $y'(0)=\cos(\theta)\sin(\phi)$ and $z'(0)=\cos(\phi)$. This shows that these are all the geodesics in $\Nil(\tau)$ passing through the origin with unit length, except for the horizontal ones, which are straight lines given by $t\mapsto(\cos(\theta)t,\sin(\theta)t,0)$ and correspond to the limit value of the parameter $\phi=\frac{\pi}{2}$.

Given $R>0$, we are interested in calculating the maximum height of $B_R(0)$, the geodesic ball of radius $R$ in $\Nil(\tau)$ centered at the origin. This is equivalent to find a value of $\phi\in[0,\pi]$ maximizing $z(R)$, where $z(t)$ is the function given by~\eqref{eqn:geodesics}. It is not difficult to prove that $z(R)>0$ if and only if $\phi\in[0,\frac{\pi}{2}[$. Since $\lim_{\phi\to\frac{\pi}{2}}z(R)=0$, we conclude that the maximum is attained for some value of $\phi\in[0,\frac{\pi}{2}[$, and we will restrict ourselves to this interval. After considering the change of variable $s=2\tau R\cos(\phi)\in\ ]0,2\tau R]$, we can reduce the problem to maximize the real-valued function $\zeta_R:\ ]0,2\tau R]\to\R$ given by
\[\zeta_R(s)=z(R)=\frac{s(s^2+4\tau^2R^2)+(s^2-4\tau^2R^2)\sin(s)}{4\tau s^2}.\]
In order to get the critical points of $\zeta_R$, we calculate
\begin{equation}\label{eqn:geodesics2}
\zeta_R'(s)=\frac{s(s^2-4\tau^2R^2)(1+\cos(s))+8\tau^2R^2\sin(s)}{4\tau s^3}.
\end{equation}
The equation $\zeta'_R(s)=0$ has two different kinds of solutions:
\begin{itemize}
 \item On the one hand, the values $s\in\ ]0,2\tau R]$ satisfying $\cos(s)=-1$. Observe that for such a value $s$, one gets that $\sin(s)=0$ so $\zeta_R(s)=\frac{s^2+4\tau^2R^2}{4\tau s}$. This last expression is a decreasing function of $s$ for $s\in[0,2\tau R]$, which implies that, among all critical values of $\zeta_R$ with $\cos(s)=-1$, the one where $\zeta_R$ has a greater value is the smaller one, i.e., $s=\pi$. Note that this only makes sense for $2\tau R>\pi$; otherwise there are no critical values of $\zeta_R$ with $\cos(s)=-1$.
 \item On the other hand, if $s\in\ ]0,2\tau R]$ is a critical value of $\zeta_R$ such that $\cos(s)\neq-1$, we can deduce from making~\eqref{eqn:geodesics2} equal zero that
\begin{equation}\label{eqn:geodesics3}
\tan\left(\frac{s}{2}\right)=\frac{\sin(s)}{1+\cos(s)}=\frac{s(4\tau^2R^2-s^2)}{8\tau^2R^2}.
\end{equation}
Note that equation~\eqref{eqn:geodesics3} has many solutions for $R$ big enough. Moreover, it allows us to work out
\begin{equation}\label{eqn:geodesics4}
\begin{aligned}
\sin(s)&=\frac{2\tan(\frac{s}{2})}{1+\tan^2(\frac{s}{2})}=\frac{16\tau^2R^2 s \left(4\tau^2R^2-s^2\right)}{16\tau^4R^4 \left(s^2+4\right)-8\tau^2 R^2
   s^4+s^6},\\
\cos(s)&=\frac{1-\tan^2(\frac{s}{2})}{1+\tan^2(\frac{s}{2})}=\frac{128\tau^4R^4}{16\tau^4R^4 \left(s^2+4\right)-8\tau^2R^2 s^4+s^6}-1.
\end{aligned}
\end{equation}
Taking derivatives in~\eqref{eqn:geodesics2} and using~\eqref{eqn:geodesics4}, the second derivative of $\zeta_R$ at a critical point $s$ such that $\cos(s)\neq-1$ can be written as
\begin{align*}
\zeta_R''(s)&=\frac{\left(4\tau^2R^2 \left(s^2-6\right)-s^4\right) \sin (s)+8\tau^2 R^2 s (2
   \cos (s)+1)}{4\tau s^4}\\
&=\frac{2\tau R^2(4\tau^2R^2-s^2)^2+96\tau^3 R^4}{s^3(4\tau^2R^2-s^2)^2+64\tau^4R^4s}>0.
\end{align*}
In particular, $\zeta_R$ does not have a (local) maximum at $s$.
\end{itemize}
As a consequence of this discussion, the maximum value of $\zeta_R(s)$ is attained either at $s=\pi$ (for $2\tau R>\pi$) or at the extremal value $s=2\tau R$ (we recall that the other extremal value $s=0$ is discarded since $\lim_{s\to 0}\zeta_R(s)=0$). Taking into account that $\zeta_R(\pi)=\frac{\pi^2+4\tau^2R^2}{4\tau\pi}$ and $\zeta_R(2\tau R)=R$, we realize that $\zeta_R(\pi)>\zeta_R(2\tau R)$ if and only if $2\tau R>\pi$. Hence we get a following sharp approximation of the spheres by cylinders.

\begin{lemma}\label{lemma:geodesic-balls-nil}
Given $R>0$, let $B_R(0)$ be the geodesic ball in $\Nil(\tau)$ centered at the origin and let $D_R(0)=\{(x,y)\in\R^2:x^2+y^2<R^2\}$.
\begin{itemize}
 \item[(a)] If $R\leq\frac{\pi}{2\tau}$, then $B_R(0)\subset D_R(0)\times]-R,R[$. 
 \item[(b)] If $R>\frac{\pi}{2\tau}$, then $B_R(0)\subset D_R(0)\times]-\frac{\pi^2+4\tau^2R^2}{4\tau\pi},\frac{\pi^2+4\tau^2R^2}{4\tau\pi}[$.
\end{itemize}
\end{lemma}

Using this estimate we are able to give an upper bound on the volume  growth of geodesic balls in $\Nil(\tau)$. We would like to point out that it seems hard to get an explicit computation of their volume due to the difficulties coming from conjugate values along geodesics. Some explicit results for small radii are given in \cite{JPP} and the references therein.

Let us consider the function $d:\R^3\to\R$ such that $d(p)$ is the distance in $\Nil(\tau)$ from $p$ to the origin, i.e., $B_R(0)=\{p\in\R^3:d(p)<R\}$. Given $\alpha>0$, we will also define the function 
\begin{equation}\label{eqn:distance-delta}
\delta_\alpha:\R^3\to\R,\qquad \delta_\alpha(p)=\max\left\{\sqrt{x^2+y^2},\tfrac{1}{\alpha}\sqrt{|z|}\right\}.
\end{equation}
Next lemma shows that $d$ and $\delta_\alpha$ are equivalent, away from the origin, in the sense of distances. Note that the ball of radius $R$ for $\delta_\alpha$ is given by $\{p\in\R^3:\delta_\alpha(p)<R\}=D_R(0)\times\ ]-\alpha^2R^2,\alpha^2R^2[$, which motivates this definition in view of item (b) in Lemma~\ref{lemma:geodesic-balls-nil}.

\begin{lemma}\label{lemma:equivalence-distances}
Given $\alpha>0$, there exist constants $M,m>0$ such that
\[m\, d(p)\leq\delta_\alpha(p)\leq M\, d(p),\]
for all $p\in\Nil(\tau)$ with $d(p)>\frac{\pi}{2\tau}$.
\end{lemma}

\begin{proof}
Let us suppose that $p=(x,y,z)\in\R^3$ is such that $R=d(p)>\frac{\pi}{2\tau}$. Then Lemma~\ref{lemma:geodesic-balls-nil} ensures that $p\in D_R(0)\times]-\frac{\pi^2+4\tau^2R^2}{4\tau\pi},\frac{\pi^2+4\tau^2R^2}{4\tau\pi}[$, so $x^2+y^2<R^2$ and $|z|<\frac{\pi^2+4\tau^2R^2}{4\tau\pi}$. From~\eqref{eqn:distance-delta} we get that $\delta_\alpha(p)\leq MR=M\, d(p)$, for some constant $M$ not depending on $p$.

In order to prove the other inequality, let us consider the Carnot-Carath\'{e}odory distance $d_{\mathrm{CC}}(p)$, defined as the infimum of the lengths of horizontal curves in $\Nil(\tau)$ joining $p$ and the origin. As the infimum is taken over horizontal curves, it is obvious that $d\leq d_{\mathrm{CC}}$, so we will prove that there exists $K>0$ such that $d_{\mathrm{CC}}(p)\leq K\, \delta_\alpha(p)$ for all $p$, and we will be done. Observe that both $d_{\mathrm{CC}}$ and $\delta_\alpha$ are homogeneous of degree 1 with respect to the dilations $(x,y,z)\mapsto(\lambda x,\lambda y,\lambda^2 z)$, i.e., they satisfy $d_{\mathrm{CC}}(\lambda x,\lambda y,\lambda^2 z)=\lambda d_{\mathrm{CC}}(x,y,z)$ and $\delta_\alpha(\lambda x,\lambda y,\lambda^2 z)=\lambda \delta_\alpha(x,y,z)$ for all $(x,y,z)\in\R^3$ and $\lambda\geq 0$. Hence, by using such dilations, we only need to prove that $d_{\mathrm{CC}}(p)\leq K\, \delta_\alpha(p)$ for all points $p\in\R^3$ with $d_{\mathrm{CC}}(p)=1$. This last assertion easily follows from the compactness of the unit sphere of $d_{\mathrm{CC}}$ and the continuity of $\delta_\alpha$ (this is a standard argument, see~\cite[Section 2.2]{Cap}).
\end{proof}

\begin{remark}
An inequality of the type $m\, d(p)\leq \delta_\alpha(p)$ is true for all $p\in\Nil(\tau)$, but the inequality $\delta_\alpha(p)\leq M\, d(p)$ is not valid in general since the quotient $\frac{\delta_\alpha}{d}$ is not bounded from above when one approaches to the origin. In fact, the statement of Lemma~\ref{lemma:equivalence-distances} is still valid after substituting $\frac{\pi}{2\tau}$ by any other positive real number.
\end{remark}

\begin{proposition}
\label{volume-geod-nil}
Geodesic balls in $\Nil(\tau)$ have quartic volume growth in the sense that, fixing $p\in\Nil(\tau)$, the function $R\mapsto R^{-4}\vol(B_R(p))$ is bounded between two positive constants independent of $p$, when $R$ is bounded away from zero.
\end{proposition}

\begin{proof}
Since $\Nil(\tau)$ is homogeneous, we can assume that $p=0$. For any $R>0$, let us define  $C_R=D_R(0)\times]-R^2,R^2[$. Lemma~\ref{lemma:equivalence-distances} with $\alpha=1$ yields the existence of $M,m>0$ such that $C_{mR}\subset B_R(0)\subset C_{MR}$ for all $R>\frac{\pi}{2\tau}$. Since the volume form in $\Nil(\tau)$ coincides with the Euclidean volume form in $\R^3$ (i.e., the identity map $\Nil(\tau)\to\R^3$ is volume-preserving) we conclude that $\vol(C_R)=2\pi R^4$ for all $R>0$, so $2\pi m^4R^4\leq\vol(B_R(0))\leq 2\pi M^4R^4$, and the statement follows.
\end{proof}

\subsection{Geodesic balls in $\PSL$}\label{sec:geodesic-balls-psl}
Let us first observe that there exist four kinds of geodesics in $\mathbb{E}(\kappa,\tau)$, $\kappa<0$, $\tau\neq 0$, depending on the nature of $\pi\circ\gamma$, i.e., the curve $\pi\circ\gamma$ can be a geodesic, a circle, a horocycle, or a hypercycle in $\H^2(\kappa)$. Up to a rotation about the $z$-axis, we will chose the geodesic $\gamma=(x,y,z)$ with unit speed and $\gamma(0)=(0,0,0)$ so that $x'(0)=0$ and $y'(0)\geq 0$. Hence $\gamma$ lies in one of the following families of examples:
\begin{enumerate}
 \item Horizontal geodesics (projecting onto a geodesic of $\mathbb{H}^2(\kappa)$)
\[
\gamma(t)=\left(0,\tfrac{2}{\sqrt{-\kappa}}\tanh(\tfrac{\sqrt{-\kappa}}{2} t),0\right),\]
from where $\gamma'(0)=(0,1,0)$.
 \item Elliptic geodesics (projecting onto a circle of $\mathbb{H}^2(\kappa)$)
 \begin{align*}
x(t)&=\frac{4a(\kappa a^2-4)(1-\cos(mt))}{16+\kappa^2a^4+8\kappa a^2\cos(mt)},\\
y(t)&=\frac{4a(\kappa a^2+4)\sin(mt)}{16+\kappa^2a^4+8\kappa a^2\cos(mt)},\\
z(t)&=\frac{4+a^2(8\tau^2-\kappa)}{\sqrt{(4-\kappa a^2)^2+64a^2\tau^2}}\,t+\frac{4\tau}{\kappa}\arctan\left(\frac{-\kappa a^2\sin(mt)}{4+\kappa a^2\cos(mt)}\right),
\end{align*}
where $0\leq a<\frac{2}{\sqrt{-\kappa}}$ is arbitrary, and we take the auxiliary parameter
\[m=\frac{2(4+\kappa a^2)\tau}{\sqrt{(4-\kappa a^2)^2+64a^2\tau^2}}.\]
Note that the initial condition $\gamma'(0)$ is given by
\[\gamma'(0)=\left(0,\frac{8a\tau}{\sqrt{(4-\kappa a^2)^2+64a^2\tau^2}},\frac{4-\kappa a^2}{\sqrt{(4-\kappa a^2)^2+64a^2\tau^2}}\right).\]
We remark that the vertical geodesic $\gamma(t)=(0,0,t)$ is obtained for $a=0$.
\item Parabolic geodesics (projecting onto a horocycle of $\mathbb{H}^2(\kappa)$)
 \begin{align*}
x(t)&=\frac{-2\sqrt{-\kappa}\tau^2 t^2}{4\tau^2-\kappa(1+\tau^2t^2)},\\
y(t)&=\frac{2\tau\sqrt{4\tau^2-\kappa}\, t}{4\tau^2-\kappa(1+\tau^2t^2)},\\
z(t)&=\frac{\sqrt{4\tau^2-\kappa}}{\sqrt{-\kappa}}\,t+\frac{4\tau}{\kappa}\arctan\left(\frac{\tau\sqrt{-\kappa}}{\sqrt{4\tau^2-\kappa}}\,t\right).
\end{align*}
If follows that the tangent vector at the origin is
\[\gamma'(0)=\left(0,\frac{2\tau}{\sqrt{4\tau^2-\kappa}},\frac{\sqrt{-\kappa}}{\sqrt{4\tau^2-\kappa}}\right)\]
\item Hyperbolic geodesics (projecting onto a hypercycle of $\mathbb{H}^2(\kappa)$)
 \begin{align*}
x(t)&=\frac{4a\sinh^2(mt)}{4+\kappa a^2\cosh^2(mt)},\\
y(t)&=\frac{a\sqrt{-\kappa a^2-4}\sinh(2mt)}{4+\kappa a^2\cosh^2(mt)},\\
z(t)&=\frac{4\tau^2-\kappa}{\kappa\sqrt{1+a^2\tau^2}}\,t+\frac{4\tau}{\kappa}\arctan\left(\frac{2\tanh(mt)}{\sqrt{-\kappa a^2-4}}\right),
\end{align*}
for any choice of $a>\frac{2}{\sqrt{-\kappa}}$, where $m=\frac{\tau\sqrt{-\kappa a^2-4}}{2\sqrt{a^2\tau^2+1}}$. Note that
\[\gamma'(0)=\left(0,\frac{a\tau}{\sqrt{1+a^2\tau^2}}, \frac{1}{\sqrt{1+a^2\tau^2}}\right).\]
\end{enumerate}

\begin{remark} This classification does not extend to the case $\kappa<0$ and $\tau=0$, because the geodesics of $\E(\kappa,\tau)=\H^2(\kappa)\times\R$ are just the product of geodesics of each factor.
\end{remark}

We shall now estimate the maximum height of the geodesic ball $B_R(0)$ in $\E(\kappa,\tau)$ when $\kappa<0$. 

\begin{lemma}\label{lemma:geodesic-balls-psl}
Given $R>0$, let $B_R(0)$ be the geodesic ball of radius $R$ in $\E(\kappa,\tau)$, $\kappa<0$, centered at the origin, and let $D_R(0)=\{(x,y)\in\R^2:x^2+y^2<\frac{4}{-\kappa}\tanh^2(\frac{\sqrt{-\kappa}}{2}R)\}$ be the corresponding disk in $\H^2(\kappa)$. Then there exists $M>0$, independent of $R$, such that $B_R(0)\subset D_R(0)\times\,]-MR,MR[$. 
\end{lemma}

\begin{proof}
Reasoning in the same way we have done in $\Nil(\tau)$, this is equivalent to estimate $z(R)$ for a unit-speed geodesic $\gamma=(x,y,z)$ with $\gamma(0)=(0,0,0)$. If $\tau=0$, then the statement is trivial (clearly we can take $M=1$), so we will suppose that $\tau\neq 0$. Taking into account the explicit expressions for $z(t)$ above, it is clear that the geodesic maximizing $z(R)$ cannot be horizontal and, in the other cases, we can estimate the $\arctan$-term by $\pm\frac{\pi}{2}$. Since the remaining term is linear in $t$, it suffices to get a bound on its coefficient not depending on the parameter $a$. It is not difficult to check that
\begin{align*}
 z(R)&\leq\left(1-\frac{1}{\kappa}(8\tau^2-\kappa)\right)R-\frac{2\pi\tau}{\kappa},&&\text{(Elliptic case)}\\
 z(R)&\leq\frac{\sqrt{4\tau^2-\kappa}}{\sqrt{-\kappa}}R-\frac{2\pi\tau}{\kappa},&&\text{(Parabolic and hyperbolic cases)}
\end{align*}
from where the statement follows.
\end{proof}

\section{Minimal graph equation in $\E(\kappa, \tau)$ and examples}
\label{minimal-graph-section}

Given a domain $\Omega\subset\M^2(\kappa)$ and a function $u\in\mathcal{C}^2(\Omega)$  we define the graph of $u$ (with respect to the zero section $(x,y)\mapsto(x,y,0)$) as the surface 
\[\Sigma=\{(x,y,u(x,y)):(x,y)\in\Omega\}.\]
It is well-known that the mean curvature $H(u)$ of the  graph $\Sigma\subset\E(\kappa,\tau)$ is given as a function on  $\M^2(\kappa)$ by the following divergence-type expression:
\begin{equation}\label{eqn:H}
H(u)=\frac{1}{2}\div \left(\frac{Gu}{\sqrt{1+\|Gu\|^2}}\right),
\end{equation}
where the divergence and the norm are computed in $\M^2(\kappa)$, and $Gu$ is a vector field on $\Omega$ given in coordinates by $Gu=(\frac{u_x}{\lambda}+\tau y)\frac{\partial_x}{\lambda}+(\frac{u_y}{\lambda}-\tau x)\frac{\partial_y}{\lambda}$. The most important feature of $Gu$ is the fact that it can be expressed as $Gu=\nabla u+Z$, where $\nabla u$ is the gradient of $u$ in $\M^2(\kappa)$, and $Z=\tau y\frac{\partial_x}{\lambda}-\tau x\frac{\partial_y}{\lambda}$ is a vector field in $\M^2(\kappa)$ independent of $u$. We will also denote $W_u=\sqrt{1+\|Gu\|^2},$ as usual.

\begin{remark}\label{rmk:generalization}
Equation~\eqref{eqn:H} is one of the keystones of our arguments below, and it is still valid in the more general scenario of Killing submersions, i.e., in an orientable 3-manifold $\E$ that admits a Riemannian submersion $\pi:\E\to M$, being $M$ a surface, such that the fibers of $\pi$ are the integral curves of a unit Killing vector field $\xi$. After choosing an initial smooth section $F_0:M\to\E$ transversal to the fibers, we can understand graphs as surfaces parametrized by $F_u:\Omega\subset M\to\E$, given by $F_u(p)=\phi_{u(p)}(F_0(p))$, being $u\in C^\infty(\Omega)$ and $\{\phi_t\}_{t\in\R}$ the 1-parameter group of isometries associated to $\xi$. It turns out that the mean curvature $H(u)$ of $F_u$, as a function on $M$, satisfies~\eqref{eqn:H}, where $Gu=\nabla u+Z$ for some vector field $Z$ on $\Omega$~\cite{LeeMan}. In that sense, some of our results extend without changes to the Killing-submersion setting (see Lemmas~\ref{lemma:area-estimate} and~\ref{lemma:area-estimate2}).
\end{remark}

Next we will briefly describe some examples of minimal surfaces in $\Nil(\tau)$.
\begin{enumerate}
 \item\textbf{Planes.} In our model of $\Nil(\tau)\equiv\R^3$, all affine planes are minimal. On the one hand, vertical planes (i.e., those projecting to a geodesic in $\R^2$) are flat and admit two foliations: one by vertical geodesics and other by horizontal geodesics. On the other hand, if $\Sigma$ is a non-vertical plane, then $\Sigma$ is a horizontal umbrella (i.e., the union of all horizontal geodesics passing through a point $p\in\Sigma$), and has negative Gauss curvature.

 \item\label{catenoid-examples}\textbf{Vertical catenoids.} Let us briefly explain how the equation of catenoids is deduced (see also~\cite{FMP}). Given the parametrization of a rotationally invariant surface
\[\Phi(t,s)=(r(t)\cos(s),r(t)\sin(s),h(t)), \qquad (t,s)\in\Omega\subset\R^2,\]
we can reparametrize it in such a way that there exists an auxiliary function $\alpha(t)$ satisfying  $h'(t)=\cos(\alpha(t))$ and $r'(t)\sqrt{1+\tau^2r(t)^2}=\sin(\alpha(t))$. Using this, $\Phi$ has constant mean curvature $H$ if and only if the following system of \textsc{ODE} is satisfied:
\begin{equation}\label{eqn:H-system}
 \left\{\begin{array}{l}
  h'(t)=\cos(\alpha(t)),\\
  r'(t)=\frac{\sin(\alpha(t))}{\sqrt{1+\tau^2r(t)^2}},\\
  \alpha'(t)=\frac{\cos(\alpha(t))+2Hr(t)}{r(t)\sqrt{1+\tau^2r(t)^2}}.
 \end{array}\right.
\end{equation}
The quantity $E=r\cos(\alpha)+Hr^2$ is constant along any solution of~\eqref{eqn:H-system}. If $H=E=0$, we get that $z$ must be constant. If $H=0$ and $E\neq 0$, we reparameterize~\eqref{eqn:H-system} by taking $r=\sqrt{x^2+y^2}$ as a variable to obtain a $1$-parameter family of \emph{catenoids} depending on the parameter $E>0$, given by:
\[h(r)=\pm\int_E^r\frac{E\sqrt{1+\tau^2s^2}}{\sqrt{s^2-E^2}}\df s,\qquad r\geq E.\]
This means that half of the catenoid is a graph over the exterior domain $r\geq E$ with zero boundary values. We also observe that $r$ is the arc-length parameter in $\R^2$, so the height growth of the catenoids is linear. Moreover catenoids have negative Gauss curvature.

\item\label{graph-examples}\textbf{Graphs.} Fern\'{a}ndez and Mira~\cite{FerMir} showed that there exists a vast family of entire minimal graphs in $\Nil(\tau)$, namely, they can associate to each holomorphic quadratic differential $Q$ on $\C$ or $\mathbb{D}=\{z\in\C:|z|<1\}$ a $2$-parameter family of entire minimal graphs with Abresch-Rosenberg differential $Q$. The only restriction is $Q\neq 0$ if the domain is $\C$.

\begin{itemize}
 \item Figueroa, Mercuri and Pedrosa in~\cite{FMP} classified minimal graphs in $\Nil(\tau)$ invariant by a 1-parameter group of left-invariant isometries. Such surfaces  are given by the graph of the function
\begin{equation}\label{eqn:FMP-vertical}
f_\theta(x,y)=\tau xy+\frac{\sinh(\theta)}{4\tau}\left[2\tau y\sqrt{1+4\tau^2y^2}+\arcsinh(2\tau y)\right],
\end{equation}
for any $\theta\in\R$.

\item Cartier~\cite[Corollary 3.8]{C} proved that  
there are non-zero minimal graphs on any wedge of $\R^2$ of angle $ ]0, \pi[\,$, with zero boundary values. The techniques in the construction involve the deformation of a horizontal umbrella. The second author, Sa Earp and Toubiana~\cite{NST} proved that, for any wedge $S$ with vertex at the origin and angle $\theta\in]\frac{\pi}{2},\pi[$, there exists a non-zero minimal graph  over $S$, with zero boundary value. Here the proof is based on classical \textsc{PDE}'s theory joint with a suitable construction of barriers.

 \item Daniel~\cite[Examples 8.4 and 8.5]{Dan2} constructed entire minimal graphs of the form $z=xf(y)$ for some real function $f$ growing linearly at $\pm\infty$. As the Figueroa-Mercuri-Pedrosa examples, they are parabolic.
\end{itemize}
\end{enumerate}

\section{Extrinsic (Spherical) area growth}
\label{sec:spherical-area}
Let us consider a minimal graph $\Sigma\subset\E(\kappa,\tau)$, $\kappa\leq 0$, over an unbounded domain $\Omega\subset\M^2(\kappa)$ given by a function $u:\Omega\rightarrow\R$. We will assume that $\partial\Omega$ is piecewise regular and in each of its regular arcs, the function $u$ either takes continuous boundary values or has $\pm\infty$ limit value. It is well-known that if $u$ takes $\pm\infty$ limit value along a curve $\gamma$, then $\gamma$ must be a geodesic arc.

Given a point $p_0\in\E(\kappa,\tau)$, we are interested in estimating the area of the intersection of $\Sigma$ with $B_R(p_0)$, the geodesic ball of $\E(\kappa,\tau)$ of radius $R,$ centered at $p_0.$ Given  a continuous positive increasing function $f:\R\to\R^+$, the surface $\Sigma$ has \emph{extrinsic area growth} of order at least $f(R)$ (resp. at most $f(R)$) if 
\begin{equation*}
\liminf_{R\to\infty}\frac{\area(\Sigma\cap B_R(p_0))}{f(R)}>0 \qquad\left(\text{resp. }\limsup_{R\to\infty}\frac{\area(\Sigma\cap B_R(p_0))}{f(R)}<\infty\right).
\end{equation*}
This definition does not depend on $p_0$, so we will assume that $p_0=0$. When $f(R)$ is a polynomial of degree $k,$ we say that the extrinsic area growth is of order at least (or at most) $k$.

The following results will give estimates of the extrinsic area growth in terms of quantities computed in the base $\Omega\subset\M^2(\kappa)$. Given $R>0$, for simplicity we will denote $\Omega(R)=\Omega\cap D_R(0)$ and $\Omega(R_1,R_2)=\Omega(R_2)\smallsetminus\overline{\Omega(R_1)}$. The proof of Lemma~\ref{lemma:area-estimate} is inspired by the work of Elbert and Rosenberg (see~\cite[Lemma 4.1]{ER}).

\begin{lemma}\label{lemma:area-estimate}
Let $\Sigma\subset\E(\kappa,\tau)$ be a minimal graph given by a function $u\in C^\infty(\Omega)$, where $\Omega\subset\M^2(\kappa)$. Given $R>0$, let us suppose that we can decompose $\partial\Omega(R)=\Lambda(R)\cup\Gamma(R)\cup\Theta(R)$, where $u$ takes limit values $\pm\infty$ along $\Lambda(R)\subset\partial\Omega$, $u$ has continuous boundary values along $\Gamma(R)\subset\partial\Omega$, and $\Theta(R)=\Omega\cap\partial\Omega(R)$. We will also assume that $B_R(0)\subset D_R(0)\times[-h(R),h(R)]$, for some positive function $h$. Then the following area estimate holds:
\begin{align*}
\area(\Sigma\cap B_R(0))\leq\area(\Omega(R))+\int_{\Omega(R)}|Z|+h(R)\Length(\Theta(R)\cup\Lambda(R))+\int_{\Gamma(R)}|u|.
\end{align*}
\end{lemma}

\begin{proof} 
Since $W=\sqrt{1+|Gu|^2}\in C^\infty(\Omega)$ is the area element of $\Sigma$ in the base domain $\Omega$ through the projection $\pi:\E(\kappa,\tau)\to\M^2(\kappa)$, we get that
\begin{equation}\label{eqn:bound1}
\begin{aligned}
\area(\Sigma\cap B_R(0))&=\int_{\pi(\Sigma\cap B_R(0))}W\leq\int_{\Omega(R)\cap\{|u|\leq h(R)\}}W\\
&=\int_{\Omega(R)\cap\{|u|\leq h(R)\}}\frac{|G u|^2}{W}+\int_{\Omega(R)\cap\{|u|\leq h(R)\}}\frac{1}{W}.
\end{aligned}
\end{equation}
Since $W\geq 1$ and $\Omega(R)\cap\{|u|\leq h(R)\}\subset\Omega(R)$, the second term in the RHS of~\eqref{eqn:bound1} satisfies
\begin{equation}\label{eqn:bound15}
\int_{\Omega(R)\cap\{|u|\leq h(R)\}}\frac{1}{W}\leq\int_{\Omega(R)}1=\area(\Omega(R)).
\end{equation}
In order to estimate the first term in the RHS of~\eqref{eqn:bound1}, we  fix $\delta>0$ and define the following auxiliary functions over $\overline\Omega$:
\begin{align*}
u_R(x)&=\begin{cases}
h(R)&\text{if }u(x)>h(R),\\
u(x)&\text{if }|u(x)|\leq h(R),\\
-h(R)&\text{if }u(x)<-h(R),
\end{cases}&
\phi_R(x)&=\begin{cases}
1&\text{if }r(x)<R,\\
\frac{(1+\delta)R-r(x)}{\delta R}&\text{if }R\leq r(x)\leq(1+\delta)R,\\
0&\text{if }(1+\delta)R<r(x),
\end{cases}
\end{align*}
where  $r(x)$ denotes the distance to the origin in $\M^2(\kappa)$.
Observe that the cut-off function $\phi_R$ is such that $0\leq\phi_R\leq 1$. By decomposing $Gu=\nabla u+Z$ and using that $\nabla u_R=\nabla u$ if $|u|<h(R)$, and $\nabla u_R=0$ if $|u|>h(R)$, we get
\begin{align}
\int_{\Omega(R)\cap\{|u|\leq h(R)\}}\frac{|Gu|^2}{W}&\leq \int_{\Omega((1+\delta)R)\cap\{|u|\leq h(R)\}}\phi_R\frac{|Gu|^2}{W}\notag\\
&=\int_{\Omega((1+\delta)R)\cap\{|u|\leq h(R)\}}\phi_R\frac{\langle Z,Gu\rangle}{W}+\int_{\Omega((1+\delta)R)}\phi_R\frac{\langle\nabla u_R,Gu\rangle}{W}\notag\\
&\leq\int_{\Omega((1+\delta)R)}\left(|Z|+\phi_R\frac{\langle\nabla u_R,Gu\rangle}{W}\right).\label{eqn:bound2}
\end{align}
In the last step we have used the Cauchy-Schwarz inequality, as well as the fact that $\phi_R\frac{|Gu|}{W}\leq 1$. To get rid of the last summand in~\eqref{eqn:bound2}, we integrate the following identity in $\Omega((1+\delta)R)$:
\begin{equation}\label{eqn:bound11}
0=\phi_R u_R\,\div\left(\frac{Gu}{W}\right)=\div\left(\phi_Ru_R\frac{Gu}{W}\right)-\frac{\phi_R\langle\nabla u_R,Gu\rangle}{W}-\frac{u_R\langle\nabla\phi_R,Gu\rangle}{W}.
\end{equation}
Stokes theorem yields
\begin{equation}\label{eqn:bound10}
\int_{\Omega((1+\delta)R)}\frac{\phi_R\langle\nabla u_R,Gu\rangle}{W}=\int_{\partial\Omega((1+\delta)R)}\phi_Ru_R\frac{\langle Gu,\eta\rangle}{W}-\int_{\Omega((1+\delta)R)}\frac{u_R\langle\nabla\phi_R,Gu\rangle}{W},
\end{equation}
where $\eta$ denotes an outer unit conormal vector field to $\Omega((1+\delta)R)$ along its boundary. We will now estimate the two terms in the RHS of~\eqref{eqn:bound10}. For the first one, we notice that the integral over $\partial\Omega((1+\delta)R)$ can be decomposed in integrals over $\Lambda((1+\delta)R)$ and $\Gamma((1+\delta)R)$, because $\phi_R$ vanishes at the rest of points of $\partial\Omega((1+\delta)R)$. Hence, using Cauchy-Schwarz inequality and the fact that $|\phi_R|\leq1$, we obtain
\begin{equation}\label{eqn:bound20}
\begin{aligned}
\int_{\partial\Omega((1+\delta)R)}\phi_Ru_R\frac{\langle Gu,\eta\rangle}{W}&=\int_{\Lambda((1+\delta)R)}\phi_Ru_R\frac{\langle Gu,\eta\rangle}{W}+\int_{\Gamma((1+\delta)R)}\phi_Ru_R\frac{\langle Gu,\eta\rangle}{W}\\
\leq&\int_{\Lambda((1+\delta)R)}|u_R|+\int_{\Gamma((1+\delta)R)}|u_R|\\
\leq& h((1+\delta)R)\Length(\Lambda((1+\delta)R))+\int_{\Gamma((1+\delta)R)}|u|,
\end{aligned}
\end{equation}
where we have used that $|u_R|=h((1+\delta)R)$ along $\Lambda((1+\delta)R)$ since $u$ takes unbounded values there, and $|u_R|\leq|u|$ along $\Gamma(R)$. In order to get a bound on the second term in the RHS of~\eqref{eqn:bound10} we will use again Cauchy-Schwarz inequality and the fact that $|\nabla\phi_R|=\frac{1}{\delta R}$ on $\Omega(R,(1+\delta)R)$ and $|\nabla\phi_R|=0$ in $\Omega(R)$. We get that
\begin{equation}\label{eqn:bound7}
-\int_{\Omega((1+\delta)R)}u_R\frac{\langle\nabla\phi_R,Gu\rangle}{W}\leq\int_{\Omega(R,(1+\delta)R)}\frac{|\nabla\phi_R||u_R||Gu|}{W}\leq\frac{h(R)}{\delta R}\area(\Omega(R,(1+\delta)R)).
\end{equation}
Plugging~\eqref{eqn:bound20} and~\eqref{eqn:bound7} into~\eqref{eqn:bound10}, and combining the result with~\eqref{eqn:bound2} and~\eqref{eqn:bound15}, it suffices to take limits for $\delta\to 0$ to get the the inequality in the statement. The only non-trivial limit is $\lim_{\delta\to 0}\frac{1}{\delta R}\area(\Omega(R,(1+\delta)R))$, but it equals $\Length(\Theta(R))$ by the coarea formula.
\end{proof}

Observe that the term $\int_{\Gamma(R)}|u|$ may be useful, for instance, in the case we know that $u$ has zero (or bounded) boundary values along some components of the boundary, or when $\Gamma(R)=\emptyset$. Nonetheless we can slightly simplify the inequality in Lemma~\ref{lemma:area-estimate} by estimating $|u_R|\leq h(R)$ along $\Gamma(R)$ in~\eqref{eqn:bound20}, to obtain the following.

\begin{lemma}~\label{lemma:area-estimate2}
Under the assumptions of Lemma~\ref{lemma:area-estimate},
\begin{align*}
\area(\Sigma\cap B_R(0))\leq\int_{\Omega(R)}(1+|Z|)+h(R)\Length(\partial\Omega(R)).
\end{align*}
\end{lemma}

These estimates could be adapted to many particular situations to obtain upper bounds on the extrinsic area growth, but now we will focus in the cases we will need in the sequel. Lemmas~\ref{lemma:area-estimate} and~\ref{lemma:area-estimate2} can be clearly generalized to the Killing-submersion setting (see Remark~\ref{rmk:generalization}).

\begin{theorem}\label{thm:area-growth-minimal}
Let $\Sigma\subset\E(\kappa,\tau)$ be a minimal graph over a domain $\Omega\subset\M^2(\kappa)$, whose boundary is piecewise regular and consists of curves along which the graph either extends continuously or takes infinite limit values. Suppose that (at least) one of the following conditions holds:
\begin{enumerate}
 \item[(i)] The graph extends continuously to $\partial\Omega$ with zero boundary values.
 \item[(ii)] There exists $K>0$ such that $\Length(\partial\Omega(R))\leq K\,\Length(\partial D_R(0))$ for $R$ sufficiently large.
\end{enumerate}
Then $\Sigma$ admits the following area estimate:
\begin{enumerate}[label=(\alph*)]
	\item If $\E(\kappa,\tau)=\R^3$, then $\Sigma$ has at most quadratic extrinsic area growth.
	\item If $\E(\kappa,\tau)=\Nil(\tau)$, then $\Sigma$ has at most cubic extrinsic area growth.
 	\item If $\kappa<0$, then $\Sigma$ has at most extrinsic area growth of order $R\mapsto R\,e^{R\sqrt{-\kappa}}$.
\end{enumerate}
\end{theorem}

\begin{proof}
If $\kappa=0$ and $\tau\neq 0$, then $\Omega(R)\subset\{(x,y)\in\R^2:x^2+y^2<R^2\}$. As the metric in $\Omega\subset\R^2$ is the Euclidean one and $Z=-\tau y\partial_x+\tau x\partial_y$
\begin{equation}\label{eqn:bound4}
 \int_{\Omega(R)}(1+|Z|)\leq\int_{x^2+y^2<R^2}\left(1+\tau\sqrt{x^2+y^2}\right)=\pi R^2+\frac{2\pi\tau}{3}R^3.
\end{equation}
Lemma~\ref{lemma:geodesic-balls-nil} implies that we can choose $h(R)=CR^2$ for some $C>0$ and $R>\frac{\pi}{2\tau}$. If (i) holds then we can apply Lemma~\ref{lemma:area-estimate} with $\Lambda(R)=\emptyset$ and $\int_{\Gamma(R)}|u|=0$. Since $\Theta(R)\subset\partial D_R(0)$, we also have that $h(R)\Length(\Theta(R)\cup\Lambda(R))\leq 2C\pi R^3$, giving at most cubic area growth. If (ii) holds, then we directly apply Lemma~\ref{lemma:area-estimate2} with $h(R)\Length(\partial\Omega(R))\leq 2CK\pi R^3$, and we are done. Note that the case of $\E(\kappa,\tau)=\R^3$ is similar to this one, but taking into account that $Z=0$ and $h(R)=R$, so the estimate only gives quadratic terms.

Let us now consider the case $\kappa<0$. Likewise we compare $\Omega(R)$ and $D_R(0)\subset\H^2(\kappa)$, given by the inequality $x^2+y^2\leq\frac{4}{-\kappa}\tanh^2(\frac{\sqrt{-\kappa}}{2}R)$. Moreover $Z=-\tau y\frac{\partial_x}{\lambda}+\tau x\frac{\partial_y}{\lambda}$ so $|Z|=\tau\sqrt{x^2+y^2}$ in $\H^2(\kappa)$. By using polar coordinates, it is not difficult to show that
\begin{align*}
 \int_{\Omega(R)}(1+|Z|)&\leq\tfrac{4\pi}{-\kappa}\sinh^2\left(\tfrac{\sqrt{-\kappa}}{2}R\right)+\tfrac{2\tau}{-\kappa}\left(\tfrac{1}{\sqrt{-\kappa}}\sinh\left(\sqrt{-\kappa}R\right)-R\right).
\end{align*}
It is also straightforward to compute $\Length(\partial D_R(0))=\frac{2\pi}{\sqrt{-\kappa}}\sinh(\sqrt{-\kappa}R)$, and Lemma~\ref{lemma:geodesic-balls-psl} allows us to consider $h(R)$ as a linear function. We conclude by the same argument as in the case $\kappa=0$, and applying either Lemma~\ref{lemma:area-estimate} or Lemma~\ref{lemma:area-estimate2}.
\end{proof}

As a first consequence of Theorem \ref{thm:area-growth-minimal}, we will use the Daniel correspondence~\cite{Dan} to obtain some intrinsic area estimates for a complete constant mean curvature $H$ graph in $\E(\kappa,\tau)$. We recall that, if $4H^2+\kappa>0$, then the only complete graphs with constant mean curvature $H$ are the horizontal slices $\mathbb{S}^2\times\{t_0\}$ in $\mathbb{S}^2\times\mathbb{R}$ (see \cite{MPR}), so the next result cover all non-trivial cases.

\begin{theorem}\label{thm:area-graphs}
Let $\Sigma\subset\E(\kappa,\tau)$ be a graph with constant mean curvature $H$ such that $4H^2+\kappa\leq 0$, and suppose that $\Sigma$ is complete.
\begin{itemize}
 \item[(a)] If $\kappa+4H^2=0$, then $\Sigma$ has at most cubic intrinsic area growth.
 \item[(b)] If $\kappa+4H^2<0$, then $\Sigma$ has at most intrinsic area growth of order $R\mapsto R\,e^{R\sqrt{-\kappa-4H^2}}$.
\end{itemize}
\end{theorem}

\begin{proof}
Under the assumption of completeness, $\Sigma$ is the graph of a function defined on a domain of $\M^2(\kappa)$ whose boundary (possibly empty) consists of complete curves of constant geodesic curvature $\pm 2H$ (see~\cite[Theorem 1]{ManRod}). In particular, $\Sigma$ is simply connected and Daniel correspondence~\cite{Dan} yields the existence of a sister minimal surface $\Sigma^*$ immersed in $\E(\kappa+4H^2,\sqrt{H^2+\tau^2})$. The surface $\Sigma^*$ is isometric to $\Sigma$ and has the same angle function, so it is also complete and transversal to the vertical Killing vector field. By means of~\cite[Theorem 1]{ManRod}, this implies that $\Sigma^*$ is also a complete minimal vertical graph over some domain of $\M^2(\kappa+4H^2)$.
\begin{itemize}
 \item[(a)] If $\kappa+4H^2=0$, then $\Sigma^*$ is a complete minimal graph in $\Nil(\sqrt{H^2+\tau^2})$, and we deduce from~\cite{DanHau} that $\Sigma^*$ is entire, so it has at most cubic extrinsic area by Theorem~\ref{thm:area-growth-minimal}.
 \item[(b)] If $\kappa+4H^2<0$, then $\Sigma^*$ is the graph on a domain of $\H^2(\kappa+4H^2)$ bounded by geodesic curves. Given $R>0$, the set $\partial\Omega(R)$ consists of finitely-many geodesic segments and finitely-many arcs in $\partial D_R(0)$. Since geodesics in $\H^2(\kappa+4H^2)$ minimize length, it turns out that each of the geodesic segments in $\partial\Omega(R)$ has smaller length than the arc in $\partial D_R(0)$ connecting its two endpoints. It follows that $\Length(\partial\Omega(R))\leq\Length(\partial D_R(0))$, so $\Sigma^*$ has extrinsic area growth of order at most $R\mapsto R\,e^{\sqrt{-\kappa-4H^2}R}$ by Theorem~\ref{thm:area-growth-minimal}. 
\end{itemize}
These estimates also hold for the intrinsic area growth, since it is always bounded from above by the extrinsic one. Since the correspondence is isometric, we deduce that $\Sigma$ has the same intrinsic area growth as $\Sigma^*$, which finishes the proof.
\end{proof}

In the following subsections, we will apply our results about extrinsic area growth, in order to estimate the area of some known examples.

\subsection{ Area of horizontal umbrellas in ${\mathbb E}(\kappa,\tau)$}
 \label{ex:area-growth-horizontal}
Let us consider $\Sigma$ to be the plane $z=0$ in the model for $\E(\kappa,\tau)$, which is nothing but the horizontal umbrella centered at the origin (i.e., it is the union of all horizontal geodesics in $\E(\kappa,\tau)$ passing through the origin) and it is minimal.
\begin{itemize}
\item If $\E(\kappa,\tau)=\Nil(\tau)$, since horizontal geodesics are Euclidean geodesics, it follows that $\Sigma\cap B_R(0)=\{(x,y,0):x^2+y^2<R^2\}$ for all $R>0$. It is easy to compute
\[\area(\Sigma\cap B_R(0))=\frac{2\pi}{3\tau^2}\left((1+\tau^2R^2)^{3/2}-1\right)=\frac{2\pi}{3}\tau R^3+O(R).\]
\item In the case $\kappa<0$ (for any $\tau$), the expression of horizontal geodesics in Section \ref{sec:geodesic-balls-psl} gives
\[\Sigma\cap B_R(0)=\left\{(x,y,0)\in\mathbb{D}\left(\tfrac{2}{\sqrt{-\kappa}}\right)\times\R:x^2+y^2\leq\tfrac{4}{-\kappa}\tanh^2\left(\tfrac{\sqrt{-\kappa}}{2}R\right)\right\}.\]
The area of $\Sigma\cap B_R(0)$ can be computed easily via the parametrization $(x,y)\mapsto(x,y,0)$ and using polar coordinates. After some computations we get
\[\area(\Sigma\cap B_R(0))=2\pi \int_0^{\frac{2}{\sqrt{-\kappa}}\tanh(\frac{\sqrt{-\kappa}}{2}R)}\frac{r\sqrt{1+\tau^2r^2}}{(1+\frac{\kappa}{4}r^2)^2}\df r=\pi \frac{\sqrt{4\tau^2-\kappa}}{-\kappa\sqrt{-\kappa}}e^{\sqrt{-\kappa}R}+O(R).\]
\end{itemize}
Note that in both cases the intrinsic ball of radius $R$ centered at the origin is given by $B_R^\Sigma(0)=\Sigma\cap B_R(0)$. This is due to the fact that horizontal geodesics are always minimizing, and proves that intrinsic and extrinsic area growths coincide for horizontal umbrellas.

It is worth mentioning that umbrellas in $\Nil(\tau)$ are hyperbolic surfaces (i.e., conformally equivalent to the unit disk $\mathbb{D}$), see~\cite[Example 8.1]{Dan2}. We can generalize this idea for $\kappa\leq 0$ and $\tau\neq 0.$ In fact,  it is easy to check that the global parameterization
\[\Phi:\mathbb{D}(\sigma)\to\E(\kappa,\tau),\qquad\Phi(u,v)=\left(\frac{2u}{\tau(1-u^2-v^2)},\frac{2v}{\tau(1-u^2-v^2)},0\right)\]
is well-defined and conformal, where $\mathbb{D}(\sigma)\subset\R^2$ is the disk of radius $\sigma=\frac{1}{2\tau}(\sqrt{-\kappa+4\tau^2}-\sqrt{-\kappa})<1$ with center at the origin. If $\tau=0$, then horizontal umbrellas are nothing but horizontal sections $\M^2(\kappa)\times\{t_0\}$, which are parabolic for $\kappa\geq 0$ and hyperbolic otherwise.

\subsection{Area of Figueroa-Mercuri-Pedrosa examples}

Let $\theta\in\R$, and let $\Sigma_{\theta}$ be the entire minimal graph of the function $f_{\theta}$ given by~\eqref{eqn:FMP-vertical}. In this case we are able to compute the exact intrinsic and extrinsic area growths.

\begin{proposition}
\label{FMP-cubic}
The minimal graph $\Sigma_\theta\subset\Nil(\tau)$ satisfies the following properties:
\begin{enumerate}
 \item[(a)] $\Sigma_\theta$ has extrinsic and intrinsic cubic area growth.
 \item[(b)] $\Sigma_\theta$ is a parabolic surface.
\end{enumerate}
\end{proposition}

\begin{proof}
First of all, we observe that the global parametrization 
\[(u,v)\in\R^2\mapsto\left(\tfrac{1}{2\tau}(\cosh(\theta) u+\sinh(\theta)\cosh(v)),\tfrac{1}{2\tau}\sinh(v),\tfrac{1}{4\tau}(\cosh(\theta)u\sinh(v)-\sinh(\theta)v)\right)\]
is conformal and the metric of $\Sigma_\theta$ reads $\tfrac{1}{4\tau^2}\cosh^2(\theta)\cosh^2(v)(\df u^2+\df v^2)$ in these coordinates (see~\cite[Example 7]{Lee}). Hence $\Sigma_\theta$ is globally conformally $\C$, so it is parabolic (see also~\cite[Example 8.2]{Dan2}). Moreover, all the surfaces $\Sigma_\theta$ are intrinsically homothetic. We will prove that $\Sigma_0$ has at least cubic intrinsic area growth. From that, it follows that $\Sigma_\theta$ has at least cubic intrinsic area growth for all $\theta$, and then $\Sigma_\theta$ will have exactly cubic intrinsic and extrinsic area growth, since the extrinsic area grows faster than the intrinsic one, and it is at most cubic by Theorem~\ref{thm:area-growth-minimal}.

Via the parametrization $(x,y)\mapsto(x,y,\tau xy)$, the surface $\Sigma_0$ is isometric to $\R^2$ endowed with the metric $\df s^2=(1+4\tau^2y^2)\df x^2+\df y^2$. Given $(x,y)\in\R^2$, let us consider the curve $\alpha$ joining $(0,0)$ and $(x,y)$ consisting in two straight segments, $\alpha_1$ (joining $(0,0)$ and $(x,0)$), and $\alpha_2$ (joining $(x,0)$ and $(x,y)$). It is easy to see that $\Length(\alpha_1)=|x|$ and $\Length(\alpha_2)=|y|$ with respect to $\df s^2$. This means that the $\df s^2$-distance from $(0,0)$ and $(x,y)$ is smaller than $|x|+|y|$, so the geodesic ball $B^{\Sigma_0}_R(0)$ in $\Sigma_0\equiv(\R^2,\df s^2)$ centered at $(0,0)$ of radius $R$ contains the square $S(R)$ of vertexes $(0,\pm R)$ and $(\pm R,0)$. Since the area element for $\df s^2$ in the $(x,y)$-coordinates is $\sqrt{1+4\tau^2y^2}$, we obtain the following lower bound for the area of $B^{\Sigma_0}_R(0)$:
\begin{align*}
\area(B^{\Sigma_0}_R(0))&\geq\int_{S(R)}\sqrt{1+4\tau^2y^2}\df x\df y=4\int_{0}^{R}\int_0^{R-y}\sqrt{1+4\tau^2y^2}\df x\df y
\\
&=\frac{1}{3 \tau^2}\left(1+(2\tau^2R^2-1)\sqrt{1+4\tau^2R^2}+3\tau R\arcsinh(2\tau R)\right).
\end{align*}
Hence, $\area(B^{\Sigma_0}_R(0))\geq\frac{4}{3}\tau R^3+O(R^2)$ and we are done.
\end{proof}

Using this result, we will correct a mistake in the Bernstein theorem for horizontal minimal multigraphs in $\Nil(\frac{1}{2})$ given in~\cite{MPR}. 

Let us take the surface $\widehat\Sigma_\theta\subset\Nil(\frac{1}{2})$ parametrized by
\[(u,v)\in\R^2\mapsto\left(g_\theta(u,v),u,v+\frac{1}{2}ug_\theta(u,v)\right),\]
where $g_\theta:\R^2\to\R^2$ is given by
\[g_\theta(u,v)=v+\frac{\sinh(\theta)}{2}\left((1+u)\sqrt{1+(1+u)^2}+\arcsinh(1+u)\right).\]
Hence $\widehat\Sigma_\theta$ is an entire graph in the direction of the Killing vector field $X=E_1+yE_3$ (see~\cite[Section 5]{MPR}) and  the isometry $F:\Nil(\frac{1}{2})\to\Nil(\frac{1}{2})$ given by $F(x,y,z)=(x,y+1,z-\frac{1}{2}x)$ satisfies $F(\widehat\Sigma_\theta)=\Sigma_\theta$. In particular, $\widehat\Sigma_\theta=F^{-1}(\Sigma_\theta)$ is an entire minimal graph in the direction of $X$, and it is parabolic by Proposition~\ref{FMP-cubic}. Theorem 3 in~\cite{MPR} contains a subtle mistake in the way $\Sigma_\theta$ is discarded as a horizontal graph (it the proof, both the surface and the Killing vector field were normalized under an ambient isometry, but this yields a loss of generality). It is fixed as follows:

\begin{theorem}[Correction of Theorem 3 in~\cite{MPR}]~\label{thm:correction}
Let $\Sigma\subset\Nil(\frac{1}{2})$ be a complete minimal surface transversal to the Killing vector field $X=E_1+yE_3$. If $\Sigma$ is parabolic, then it is either a vertical plane or an invariant surface $\Sigma_\theta$, up to an ambient isometry.
\end{theorem}

\subsection{Area of ideal Scherk graphs ($\kappa<0$)} 
Let $\Omega\subset\H^2(\kappa)$ be an unbounded domain whose boundary is an ideal polygon consisting of finitely-many complete curves with alternating constant geodesic curvature $\pm 2H$, meeting at some points at the ideal boundary $\partial_\infty\H^2(\kappa)$. Let $\Sigma\subset\E(\kappa,\tau)$ be the graph  of a function $u$ defined over $\Omega$, with constant mean curvature $H$ satisfying $4H^2+\kappa<0$, and such that $u$ has boundary values $\pm\infty$ along each curve in $\partial\Omega$. Such surface  $\Sigma$ is known as an ideal Scherk graph. The existence of ideal Scherk graphs in $\H^2\times\R$ was proven by Collin and Rosenberg in \cite{CR}, under some conditions on the shape of the domain $\Omega$. In fact, the conditions were inspired by those founded by Jenkins and Serrin for a minimal graph in $\R^3$  with infinite boundary values. Analogous existence results were given by Folha and Melo \cite{FM} ($0<H<\frac{1}{2}$ in $\H^2\times\R$) and Melo \cite{Me} ($H=0$ in $\PSL$).

The aim of this section is to show that the area growth of ideal Scherk graphs is similar to the area growth of the vertical surfaces they are asymptotic to (Theorem \ref{area-scherk-theo}). 

\begin{remark}
Let $\Sigma=\pi^{-1}(\Gamma)\subset\E(\kappa,\tau)$, being $\Gamma\subset\M^2(\kappa)$ a geodesic. Then $\Sigma$ is minimal and has quadratic intrinsic area growth provided that $\kappa\leq 0$, for it is isometric to $\R^2$ endowed with the Euclidean flat metric. Let us suppose that $\Sigma$ is given by the equation $x=0$ in our model.
\begin{itemize}
 \item If  $\kappa=0,$ $\E(\kappa,\tau)=\Nil(\tau)$, we get by Lemma~\ref{lemma:equivalence-distances} ($\alpha=1$) that there exist $m,M>0$ such that $\Sigma\cap C_{mR}\subset\Sigma\cap B_R(0)\subset\Sigma\cap C_{MR}$ for $R>\frac{\pi}{2\tau}$. Here we have considered the cylinders $C_R=D_R(0)\times]\!-\!R^2,R^2[$ that satisfy $\area(\Sigma\cap C_R)=4R^3$ for all $R>0$, which implies that $\Sigma$ has cubic extrinsic area growth.
 \item If $\kappa<0$, then $B_R(0)\subset D_R(0)\times]-MR,MR[$ for some $M>0$ by Lemma~\ref{lemma:geodesic-balls-psl}, so $\area(B_R(0)\cap\Sigma)\leq\area(\Sigma\cap(D_R(0)\times]-MR,MR[))\leq 4MR^2$ and $\Sigma$ has at most quadratic extrinsic area growth. Since the intrinsic area growth is quadratic and it represents a lower bound for the extrinsic one, we deduce that the extrinsic area growth of $\Sigma$ is also quadratic.
\end{itemize}
\end{remark}

Next we will analyze the area of the projection of a complete graph $\Sigma$ in $\E(\kappa,\tau)$, which will be the key step in the proof of Theorem \ref{area-scherk-theo}. Let us mention that the angle function $\nu=\langle E_3,N\rangle$, being $N$ the upward-pointing unit normal to $\Sigma$, is such that $\int_G\nu=\area(\pi(G))$ for any region $G\subset\Sigma$, where $\pi:\E(\kappa,\tau)\to\H^2(\kappa)$ denotes the usual projection. This assertion follows from the fact that the Jacobian of $\pi_{|\Sigma}$ equals $|\nu|$, and $\nu>0$ because of our choice of the unit normal $N$.

\begin{proposition}\label{prop:finite-area}
Let $\Sigma\subset\E(\kappa,\tau)$ be a complete graph with constant mean curvature $H$ such that $4H^2+\kappa<0$, and projecting onto a domain $\Omega\subset\H^2(\kappa)$.
\begin{itemize}
 \item[(a)] If $\Sigma$ is an ideal Scherk graph and $\Omega$ has $2n$ ideal vertexes, then $\Omega$ has finite area given by
\[\area(\Omega)=\frac{2(n-1)\pi}{-\kappa-4H^2}.\]
\item[(b)] If $\Sigma$ is not an ideal Scherk graph, then $\Omega$ has infinite area.
\end{itemize}
\end{proposition}

\begin{proof}
Given $R>0$, let us consider $\Omega(R)=\Omega\cap D_R(0)$. Then $\partial\Omega(R)$ can be decomposed in three finite families of arcs, namely, those with geodesic curvature $2H$ (along which $\Sigma$ takes $+\infty$ limit boundary value), those with geodesic curvature $-2H$ (with $-\infty$ limit boundary value) and those in $\Omega\cap\partial D_R(0)$. Geodesic curvature is always computed with respect to the inner conormal vector field to $\Omega(R)$ along its boundary. Let us call $\alpha(R)$, $\beta(R)$ and $\ell(R)$ the lengths of the segments in the first, second and third family, respectively. Gauss-Bonnet theorem applied to $\Omega(R)\subset\H^2(\kappa)$ yields the following identity:
\begin{equation}\label{eqn:area-estimate-gb}
\kappa\area(\Omega(R))=2\pi-2H(\alpha(R)-\beta(R))-\kappa_g(R)\ell(R)-\Theta(R),
\end{equation}
where $\kappa_g(R)>1$ is the geodesic curvature of $\partial D_R(0)$, and $\Theta(R)$ denotes the sum of all exterior angles at the vertexes of $\partial\Omega(R)$. Now, let us consider $T=E_3-\nu N$ to be the tangent part of the vertical Killing vector field $E_3$, which satisfies $\div_\Sigma(T)=2H\nu$. By a classical application of the divergence theorem to $T$ (also known as flux formula) on $\Sigma(R)=\pi^{-1}_{|\Sigma}(\Omega(R))$, it follows that
\begin{equation}\label{eqn:area-estimate-flux}
 2H \area(\Omega(R))=\int_{\Sigma(R)}\div_{\Sigma}(T)=\alpha(R)-\beta(R)+\int_{\partial\Sigma(R)}\langle T,\eta\rangle,
\end{equation}
where $\eta$ stands for an unit conormal to $\Sigma(R)$ along its boundary.
The term $\alpha(R)-\beta(R)$ appears since we indeed apply the divergence theorem to compact subdomains of $\Sigma(R)$ uniformly converging to $\Sigma(R)$, and the angle function uniformly tends to 0 along the boundary curves with geodesic curvature $\pm 2H$. Hence, combining~\eqref{eqn:area-estimate-gb} and \eqref{eqn:area-estimate-flux}, we get that
 \begin{equation}\label{eqn:area-estimate1}
 (-\kappa-4H^2)\area(\Omega(R))=-2\pi+\Theta(R)+\kappa_g(R)\ell(R)-2H\int_{\partial\Sigma(R)}\langle T,\eta\rangle,
 \end{equation}
Observe that $\partial\Sigma(R)$ consists of curves of infinite length, but the last integral in~\eqref{eqn:area-estimate1} is finite. If we parametrize one of these curves by $\gamma:\R\to\E(\kappa,\tau)$ with unit speed and consider $J$ to be the  $\frac\pi2$-rotation in  $T\Sigma$, we have that $\langle T,\eta\rangle=\langle JT,\gamma'\rangle$. Since $\{T,JT\}$ is an orthogonal frame on $\Sigma$ with $|T|^2=|JT|^2=1-\nu^2$, we can express $1-\nu^2=\langle T,\gamma'\rangle^2+\langle JT,\gamma'\rangle^2$, and hence
 \begin{equation}\label{eqn:area-estimate2}
 \left|\int_\gamma\langle T,\eta\rangle\right|\leq\int_\gamma|\langle JT,\gamma'\rangle|=\int_\gamma\sqrt{1-\nu^2-\langle T,\gamma'\rangle^2}\leq\int_\gamma\sqrt{1-\langle T,\gamma'\rangle^2}=\Length(\pi\circ\gamma).
 \end{equation}
The last equality in~\eqref{eqn:area-estimate2} follows from the fact that $\langle T,\gamma'\rangle=\langle E_3,\gamma'\rangle$ (note that $\gamma'$ is tangent), and the fact that $|(\pi\circ\gamma)'|^2+\langle E_3,\gamma'\rangle^2=1$, which follows from decomposing $\gamma'$ in vertical and horizontal components. Hence the absolute value of the last integral in~\eqref{eqn:area-estimate1} is at most $\ell(R)$.

Let us now take limits in~\eqref{eqn:area-estimate1} when $R\to\infty$ and distinguish two cases:
\begin{itemize}
 \item[(a)] If $\Sigma$ is an ideal Scherk graph, then it is easy to prove that $\lim_{R\to\infty}\ell(R)=0$, since two successive components of $\partial\Omega$ approach exponentially in $R.$ Moreover $\lim_{R\to\infty}\kappa_g(R)=-\kappa$. On the other hand, $\Theta(R)=2n\pi$ for $R$ sufficiently large (it suffices to take $R$ such that $\partial D_R(0)$ intersects transversally all the components of $\partial\Omega$). Thus it follows from~\eqref{eqn:area-estimate1} that $(-\kappa-4H^2)\area(\Omega)=2(n-1)\pi$ and the statement follows.
 
 \item[(b)] Let us now suppose that $\Sigma$ is not an ideal Scherk graph. Then we have two possible situations: either $\Omega$ contains an arc at infinity (so it is clear that $\area(\Omega)=\infty$ and we are done), or $\partial\Omega$ consists of infinitely-many curves of geodesic curvature $\pm 2H$. In the latter case, we have proved that $\kappa_g(R)\ell(R)-2H\int_{\partial\Sigma(R)}\langle T,\eta\rangle\geq(\kappa_g(R)-2H)\ell(R)>0$, and hence~\eqref{eqn:area-estimate1} implies $(-\kappa-4H^2)\area(\Omega(R))>-2\pi+\Theta(R)$ for all $R>0$. It suffices to check that $\lim_{R\to\infty}\Theta(R)=\infty$, but this is straightforward since $\partial\Omega(R)$ contains eventually an arbitrarily large number of vertexes, and the exterior angle at each of these vertexes converges to some value, bounded away from zero, only depending on $H$.\qedhere
\end{itemize}
 \end{proof}

\begin{corollary}\label{coro:scherk-preserved}
If $\Sigma\subset\E(\kappa,\tau)$  and $\Sigma^*\subset\E(\kappa^*,\tau^* )$ are sister surfaces by the Daniel correspondence, and $\Sigma$ is an ideal Scherk graph, then so is $\Sigma^*$. Moreover, they are graphs over ideal polygons with the same number of ideal vertexes.
\end{corollary}

\begin{proof}
As mentioned in the proof of Theorem~\ref{thm:area-graphs}, if $\Sigma$ is an ideal Scherk graph, then it follows from~\cite{ManRod} that $\Sigma^*$ is a complete graph over a domain $\Omega^*=\pi(\Sigma^*)\subset\H^2(\kappa^*)$ bounded by curves of geodesic curvature $\pm 2H^*$, where $H^*$ is the mean curvature of $\Sigma^*$. Since $\Omega=\pi(\Sigma)$ has finite area by Proposition~\ref{prop:finite-area}.(a), the angle function is preserved by the correspondence, and the integral of the angle is the area of the projection, we get that $\Omega^*$ also has finite area, so it is a Scherk graph by Proposition~\ref{prop:finite-area}.(b). Since $\kappa+4H^2=\kappa^*+4{H^*}^2,$ the number of vertexes is also preserved.
\end{proof}

Finally we can prove the desired area estimate.

\begin{theorem}
\label{area-scherk-theo}
If $\Sigma\subset\E(\kappa,\tau)$ is an ideal Scherk graph, then $\Sigma$ has at most quadratic intrinsic area growth. In particular, the underlying conformal structure of $\Sigma$ is parabolic.

If $\Sigma$ is minimal, then it also has quadratic extrinsic area growth.
\end{theorem}

\begin{proof}
Let $\Sigma^*\subset\E(\kappa+4H^2,\sqrt{\tau^2+H^2})$ be the sister minimal surface by Daniel correspondence, and let us prove that $\Sigma^*$ has at most quadratic extrinsic area growth. This implies that $\Sigma^*$, and hence $\Sigma$, has at most quadratic intrinsic area growth, so we will be done. Moreover, the assertion about the conformal structure follows from a result by Cheng and Yau~\cite[Corollary 1]{CY}. 

By Corollary~\ref{coro:scherk-preserved}, the surface $\Sigma^*$ is also an ideal Scherk graph over some domain $\Omega^*\subset\H^2(\kappa+4H^2)$ with $2n$ ideal vertexes for some $n\in\N$. In order to apply Lemma~\ref{lemma:area-estimate2} to $\Sigma^*$, we first observe that $\area(\Omega^*)$ is finite and $|Z|=\sqrt{(\tau^2+H^2)(x^2+y^2)}$ is bounded so $\int_{\Omega^*(R)}(1+|Z|)$ is bounded independently on $R$. Moreover, for $R$ sufficiently large, $\partial\Omega^*(R)$ consist of $2n$ geodesic segments as well as some arcs contained in $\partial D_R(0)$. On the one hand, the length of these arcs in $\partial D_R(0)$ can be easily shown to converge to zero when $R\to\infty$. On the other hand, each geodesic segment in $\partial\Omega^*(R)$ has length at most $2R$, the diameter of $D_R(0)$, because geodesics minimize length in $\H^2(\kappa+4H^2)$. Hence $\Length(\partial\Omega^*(R))$ grows linearly. As  $\kappa+4H^2<0,$  Lemma~\ref{lemma:geodesic-balls-psl} allows us to take $h(R)$ as a linear function. By applying Lemma~\ref{lemma:area-estimate2}, we can guarantee that $\area(\Sigma^*\cap B_R(0))$ grows at most quadratically.
\end{proof}

\subsection{Area of catenoids and $k$-noids in $\H^2(\kappa)\times\R$} 
As a last family of examples, we will study symmetric $k$-noids in $\H^2(\kappa)\times\R$ with constant mean curvature $H$ such that $4H^2+\kappa<0$. Minimal $k$-noids were constructed independently by Morabito and Rodr\'{i}guez~\cite{MorRod}, and Pyo~\cite{Pyo}, though in~\cite{MorRod} the non-symmetric case is also considered. We emphasize that horizontal catenoids are recovered as $2$-noids, and were first obtained by Daniel and Hauswirth~\cite{DanHau} for $H=\frac{1}{2}$ in $\H^2\times\R$, by means of a representation formula for the Gauss map of minimal surfaces in $\Nil(\frac{1}{2})$. For the rest of values of the mean curvature, symmetric $k$-noids were obtained by Plehnert~\cite{Plehnert}. They are complete embedded surfaces in $\H^2\times\R$ with genus zero and $k$ ends, which are asymptotic to vertical cylinders over curves of geodesic curvature $-2H$. We will prove that they have at most quadratic intrinsic area growth, illustrating how our techniques can be easily adapted to conjugate Plateau constructions.

The key idea in the construction is to realize that such a $k$-noid $\Sigma_k\subset\H^2(\kappa)\times\R$ can be decomposed in $4k$ pieces which are congruent by ambient isometries. By Daniel correspondence, each piece is isometric to a minimal graph $\Sigma^*$ in $\E(\kappa+4H^2,H)$,  which is obtained by solving an improper Plateau problem. The graph $\Sigma^*$ projects onto an ideal geodesic triangle $\Delta\subset\H^2(\kappa+4H^2)$ which has a vertex at infinity (so two of its sides have infinite length), and the other two vertexes having angles  $\frac{\pi}{2}$ and $\frac{\pi}{k}$ (see~\cite[Section 3.3]{Plehnert}). The surface $\Sigma^*$ is obtained by solving the Dirichlet problem with zero boundary values along the sides sharing the $\frac{\pi}{k}$-angle, and $+\infty$ limit value along the third side of $\Delta$. Hence it is clear that $\area(\Delta)<\infty$ and $\Length(\partial\Omega(R))$ grows linearly, so Lemma~\ref{lemma:area-estimate} yields that $\Sigma^*$ has at most quadratic extrinsic (and  intrinsic) area growth. Since sister surfaces are isometric and $\Sigma_k$ consists of $4k$ pieces isometric to $\Sigma^*$, we get the following result.

\begin{theorem}
Given $k\geq 2$, the $k$-noid $\Sigma_k\subset\H^2(\kappa)\times\R$ with constant mean curvature $H$ such that $4H^2+\kappa<0$, constructed in~\cite{Plehnert}, has at most quadratic intrinsic area growth. In particular, $\Sigma_k$ is parabolic, so it is conformally equivalent to $\mathbb{S}^2$ minus $k$ points.
\end{theorem}

\section{Cylindrical area growth}
\label{sec:cylindrical-area}

Now we will introduce a different concept related to the area growth of a surface in $\E(\kappa,\tau)$. Given $x\in\M^2(\kappa)$ and $R>0$, the cylinder centered at $x_0$ of radius $R$ is the open subset $C_R(x_0)=\pi^{-1}(D_R(x_0))$. In analogy to Section \ref{sec:spherical-area}, we define the cylindrical area growth of $\Sigma$ as follows.
Given a positive increasing continuous function $f:\R\to\R^+$, the surface $\Sigma$ has {\em cylindrical area growth} of order at least $f(R)$ (resp. at most $f(R)$) if 
\begin{equation*}
\liminf_{R\to\infty}\frac{\area(\Sigma\cap C_R(x_0))}{f(R)}>0 \qquad\left(\text{resp. }\limsup_{R\to\infty}\frac{\area(\Sigma\cap C_R(x_0))}{f(R)}<\infty\right).
\end{equation*}
 When $f(R)$ is a polynomial of degree $k,$ we say that the cylindrical area growth is of order at least (resp. at most) $k$. This definition does not depend on the choice of $x_0\in\M^2(\kappa)$ either, and it is invariant under ambient isometries of $\E(\kappa,\tau)$ (note that any isometry $F$ of $\E(\kappa,\tau)$ satisfies $F(C_R(\pi(p)))=C_R(\pi(F(p))$ for all $p\in\E(\kappa,\tau)$ and $R>0$). The cylindrical area is appropriate to study the area growth of entire graphs in $\E(\kappa,\tau)$, as will shall see below.

In Example~\ref{ex:area-growth-horizontal}, it is proved that the horizontal umbrella $\Sigma_0$ given by $z=0$ has cubic cylindrical area growth in $\Nil(\tau)$, whereas it grows as $R\mapsto e^{R\sqrt{-\kappa}}$ in the case $\kappa<0$. This follows from the fact that $\Sigma_0\cap C_R(0)=\Sigma_0\cap B_R(0)$, so both cylindrical and extrinsic area growths coincide over $\Sigma$.  
Next result proves that the graphical surfaces which minimize area with free boundary over vertical cylinders are precisely horizontal umbrellas.

\begin{proposition}\label{prop:callibration}
Let $\Sigma\subset\E(\kappa,\tau)$ be the graph of a function $u:D_R(x_0)\subset\M^2(\kappa)\to\R$ which extends continuously to $\overline D_R(x_0)$, and let $\Sigma_0$ be the horizontal umbrella centered at $p_0$ such that $\pi(p_0)=x_0$. Then
\[\area(\Sigma)\geq\area(\Sigma_0\cap C_R(x_0)).\]
Equality holds if and only if $\Sigma=\Sigma_0\cap C_R(x_0)$, up to a vertical translation.
\end{proposition}

\begin{proof}
By applying an appropriate isometry of $\E(\kappa,\tau)$ and choosing $\Sigma_0$ up to a vertical translation, we may assume that $x_0=0,$  $u>0$ in $D_R(0)$, and also that $\Sigma_0$ is given by the equation $z=0$. The upward-pointing unit normal vector field to $\Sigma_0$ is
\[N_0=\frac{-\tau y E_1+\tau x E_2+E_3}{\sqrt{1+\tau^2(x^2+y^2)}}.\]
This expression extends $N_0$ to a global unit vector field in $\E(\kappa,\tau)$ with zero divergence (it is the unit normal to a foliation of $\E(\kappa,\tau)$ by minimal surfaces). Now, let $U\subset\E(\kappa,\tau)$ be the bounded region with boundary $\Sigma_0\cap C_R(0)$, $\Sigma$, and the cylinder $\partial C_R(0)$, and apply the divergence theorem to $N_0$ in $U$. Since $N_0$ is orthogonal to $\Sigma_0$ and tangent to $\partial C_R(0)$, we get
\[\int_\Sigma\langle N_0,N\rangle=\area\left(\Sigma_0\cap C_R(0)\right),\]
where $N$ denotes the upward-pointing unit normal vector field to $\Sigma$. Since $\langle N_0,N\rangle\leq 1,$ Cauchy-Schwarz inequality yields that $\area(\Sigma)\geq\int_\Sigma\langle N_0,N\rangle$, and we are done. If equality holds, then $\langle N_0,N\rangle=1$, so $N=N_0$ and $\Sigma$ differs from $\Sigma_0$ by a vertical translation.
\end{proof}

Combining Proposition \ref{prop:callibration} with the explicit cylindrical area growth of horizontal umbrellas, we get a global estimate for entire graphs.

\begin{corollary}\label{coro:cylindrical-growth}
Let $\Sigma\subset\E(\kappa,\tau)$ be an entire graph. 
\begin{itemize}
 \item[(a)] If $\E(\kappa,\tau)=\Nil(\tau)$, then $\Sigma$ has at least cubic cylindrical area growth.
 \item[(b)] If $\kappa< 0$, then $\Sigma$ has at least cylindrical area growth of order $R\mapsto e^{\sqrt{-\kappa}R}$.
\end{itemize} 
\end{corollary}

We observe that the cylindrical area growth and the extrinsic area growth of entire minimal graphs in $\Nil(\tau)$ admit estimates, from below and from above, respectively, of order 3. If we could guarantee that they coincide for some entire minimal graph $\Sigma$, then we would be able to prove that $\Sigma$ has exactly cubic extrinsic area growth. Next result shows that this is the case when we assume a restriction on the growth of the height of the entire graph.

\begin{corollary}\label{coro:height-area-estimate-nil}
Let $\Sigma\subset\Nil(\tau)$ be an entire minimal graph given by a function $u\in\mathcal{C}^\infty(\R^2)$, and assume that there exist constants $M>0$ and $\beta\geq 1$ such that $|u|\leq M(1+r^2)^\beta$, being $r$ the distance to the origin in $\R^2$. Then the extrinsic area growth of $\Sigma$ has  order at least $\frac{3}{\beta}$.

In particular, if $\beta=1$, then $\Sigma$ has exactly cubic extrinsic area growth.
\end{corollary}

\begin{proof} 
Let us consider the distance $\delta_\alpha$ defined by~\eqref{eqn:distance-delta} for some $\alpha>\sqrt{M}$. The equivalence between $d$ and $\delta_\alpha$ established in Lemma~\ref{lemma:equivalence-distances} implies that the growth of the function $R\mapsto\area(\Sigma\cap B_R(0))$ is asymptotically the same as the growth of $R\mapsto\area(\Sigma\cap(D_R(0)\times]-\alpha^2R^2,\alpha^2R^2[))$ (we recall that $D_R(0)\times\,]-\alpha^2R^2,\alpha^2R^2[$ is the ball of radius $R$ for $\delta_\alpha$). Since $\beta\geq 1$, we get that
\[\area(\Sigma\cap(D_R(0)\times]-\alpha^2R^2,\alpha^2R^2[))\geq \area(\Sigma\cap(D_{R^{1/\beta}}(0)\times]-\alpha^2R^2,\alpha^2R^2[)).\]
Now we observe that, by hypothesis, $|u|\leq M(1+R^{2/\beta})^\beta$ on $D_{R^{1/\beta}}(0)$, and $M(1+R^{2/\beta})^\beta\leq\alpha^2R^2$ for $R$ sufficiently large. It implies that $\Sigma\cap(D_{R^{1/\beta}}(0)\times]-\alpha^2R^2,\alpha^2R^2[)=\Sigma\cap C_{R^{1/\beta}}(0)$ for $R$ sufficiently large, so $\area(\Sigma\cap C_{R^{1/\beta}}(0))$ grows at least as the function $R\mapsto(R^{1/\beta})^3=R^{3/\beta}$ by Corollary \ref{coro:cylindrical-growth}, so we get the desired estimate.

If $\beta=1$, this estimate implies that $\Sigma$ has at least cubic extrinsic area growth, and we conclude by Theorem~\ref{thm:area-growth-minimal} that the extrinsic area growth is exactly cubic.
\end{proof}

Corollary~\ref{coro:height-area-estimate-nil} gives a relation between the area growth and height  of the graph. In this direction we can prove the following gradient, height and area estimates.

\begin{theorem}\label{last-area-estimate}
Let $\Sigma\subset\Nil(\tau)$ be an entire minimal graph, given by a function $u\in\mathcal{C}^\infty(\R^2)$, and consider $r=\sqrt{x^2+y^2}$. Then
\begin{itemize}
 \item[(a)] there exists a constant $B>0$ such that $|Gu|\leq B(1+r^2)$,
 \item[(b)] there exists a constant $C>0$ such that $|u|\leq C(1+r^2)^{3/2}$.
\end{itemize}
In particular, the extrinsic area growth of $\Sigma$ is at least quadratic, and at most cubic, while the cylindrical area growth of $\Sigma$ is at least cubic, and at most quartic.
\end{theorem}

The proof of Theorem~\ref{last-area-estimate} will rely on the following gradient estimate for entire spacelike graphs in Lorentz-Minkowski $3$-space $\LL^3$ with constant positive mean curvature, via the Calabi-type correspondence by Lee~\cite[Corollary 2]{Lee}.

\begin{lemma}\label{lemma:gradient-estimate-L3}
Let $\Sigma\subset\LL^3$ be an entire spacelike graph with constant mean curvature $H>0$, given by a global parametrization $(x,y)\mapsto (x,y,v(x,y))$ for a certain function $v\in\mathcal{C}^\infty(\R^2)$. Then there exists a constant $A>0$ such that
\[|\nabla v|^2\leq 1-\frac{A}{(1+r^2)^2},\]
being $r$ the distance to the origin in $\R^2$ and $\nabla v$ the usual gradient of $v$ in $\R^2$.
\end{lemma}

\begin{proof}
First we can assume that $\Sigma$ is not a ruled surface, and apply a translation in $\LL^3$ such that the origin belongs to $\Sigma$, and no straight line through the origin is contained in $\Sigma$. Notice that this normalization does not affect the estimate we are looking for. Note also that, if $\Sigma$ were a ruled surface, then the classification of ruled constant mean curvature surfaces in $\LL^3$ given in~\cite{DVVW} implies that $\Sigma$ is either minimal, or a circular cylinder, or a hyperbolic cylinder or an isoparametric surface with null curves. Since we are dealing with $H>0$ and the surface is spacelike, we conclude that $\Sigma$ is a hyperbolic cylinder, i.e., up to an isometry of $\LL^3$, we can suppose that $v(x,y)=\frac{1}{2H}\sqrt{1+4H^2x^2}$, so $|\nabla v|^2=1-(1+4H^2x^2)^{-1}$ and the statement follows. Moreover, it follows from the work of Treibergs~\cite{Treibergs} that $\Sigma$ is the boundary a convex set in $\LL^3$ (see also~\cite{ACL}) so, up to a mirror reflection with respect to $z=0$, we may also assume that $v$ is a convex function.

Following the arguments in~\cite{CY2} (see also~\cite{Treibergs}), the surface $\Sigma$ is complete with respect to its induced Riemannian metric, and the Lorentzian support function $\Phi:\LL^3\to\R$ given by $\Phi(x,y,z)=x^2+y^2-z^2$ is a proper function on $\Sigma$ satisfying the gradient estimate $|\nabla^\Sigma\Phi|^2\leq C(1+\Phi)^2$, where the gradient is computed on $\Sigma$. We will develop this inequality to get our result.

By taking into account that $\nabla^\Sigma\Phi=\overline\nabla\Phi+\langle\overline\nabla\Phi,N\rangle N$, where $\overline\nabla\Phi=2x\partial_x+2y\partial_y+2z\partial_z$ denotes the gradient of $\Phi$ in the ambient space $\LL^3$, and $N=(1-v_x^2-v_y^2)^{-1/2}(v_x\partial_x+v_y\partial_y+\partial_z)$ is a unit normal to $\Sigma$, we reach the following expression:
\begin{equation}\label{eqn:L3-estimate1}
|\nabla^\Sigma\Phi|^2=|\overline\nabla\Phi|^2+\langle\overline\nabla\Phi,N\rangle^2\geq \langle\overline\nabla\Phi,N\rangle^2=\frac{4(v-xv_x-yv_y)^2}{1-v_x^2-v_y^2}.
\end{equation}
Next we estimate the numerator in the RHS of~\eqref{eqn:L3-estimate1}, for what we observe that the intersection of the $z$-axis and the tangent line to $\Sigma$ at $(x,y,v)$ in the direction of the tangent vector $(x,y,xv_x+yv_y)$ is precisely the point $(0,0,v-xv_x-yv_y)$. Using this and the fact that $\Sigma$ is convex and does not contain a line through the origin, we get that
\begin{equation}\label{eqn:L3-estimate2}
w(x,y)\leq w\left(\frac{x}{\sqrt{x^2+y^2}},\frac{y}{\sqrt{x^2+y^2}}\right)<0,\qquad\quad \text{if }x^2+y^2\geq 1,
\end{equation}
where $w=v-xv_x-yv_y$. Since $w$ is continuous and the unit circle is compact, equation~\eqref{eqn:L3-estimate2} implies that there exists a constant $M>0$ such that $w^2=(v-xv_x-yv_y)^2\geq M$ provided that $x^2+y^2\geq 1$. Hence the gradient estimate for the support function by Cheng and Yau (see\cite[Theorem 1]{CY2} and \cite[Proposition 2]{Treibergs}) yields
\[\frac{4M^2}{1-|\nabla v|^2}\leq|\nabla^\Sigma\Phi|^2\leq C(1+\Phi)^2=C(1+x^2+y^2-v^2)^2\leq C(1+x^2+y^2)^2.\]
Equivalently, $|\nabla v|^2\leq 1-\frac{4M^2}{C}(1+r^2)^{-2}$. Though this inequality is valid for $r\geq 1$, it trivially extends for all $r\geq 0$ by possibly changing the constant $M$, so the statement follows.
\end{proof}

\begin{proof}[Proof of Theorem \ref{last-area-estimate}]
By taking into account the Calabi-type correspondence in~\cite{Lee} we can associate to $\Sigma$ a function $v\in\mathcal{C}^\infty(\R^2)$ such that $(x,y)\mapsto(x,y,v(x,y))$ defines an entire spacelike graph in $\LL^3$ with constant mean curvature $\tau$, and $v$ satisfies the relation $(1-|\nabla v|^2)(1+|Gu|^2)=1$. From Lemma~\ref{lemma:gradient-estimate-L3} we get that there exists $A>0$ such that $1+|Gu|^2=(1-|\nabla v|^2)^{-1}\leq A^{-1}(1+r^2)^2<1+A^{-1}(1+r^2)^2$, and we get item (a) by just taking $B=A^{-1/2}$.

Applying the Minkowski inequality to the expression $\nabla u=Gu-Z$, where $Z=-\tau y\partial_x+\tau x\partial_y$, we get that  $|\nabla u|\leq |Gu|+|Z|\leq B(1+r^2)+\tau r$. Hence $|\nabla u|$ grows at most quadratically in $r$, from where it is easy to see that there exists a constant $C>0$ satisfying item (b). 

The assertion about extrinsic area growth in the statement is a consequence of Corollary~\ref{coro:height-area-estimate-nil} ($\beta=\frac{3}{2}$) and Theorem~\ref{thm:area-growth-minimal}. Finally, the assertion about cylindrical area growth follows from Corollary~\ref{coro:cylindrical-growth} and a simple integration in polar coordinates using item (b):
\[\area(\Sigma\cap C_R(0))=\int_{D_R(0)}\sqrt{1+\|Gu\|^2}\leq 2\pi\int_0^Rr\sqrt{1+B^2(1+r^2)^2}\df r\leq DR^4,\]
for $r$ big enough and some constant $D>0$.
\end{proof}

As mentioned above, our estimates of the gradient give estimates on the angle function. This allows us to improve a result by Espinar~\cite[Corollary~5.2]{Espinar} on complete stable surfaces with constant mean curvature. Here, stability is understood in a strong sense (i.e., a constant mean curvature $H$ surface is said stable if it is a stable critical point of the functional $\mathcal{J}=\area-2H\,\vol$, for all normal variations of $\Sigma$ with compact support, see also~\cite{MPR}).

\begin{corollary}\label{coro:estability}
Let $\Sigma$ be an orientable complete stable surface with constant mean curvature $H$ immersed in $\E(\kappa,\tau)$, with $\tau\neq 0$ and $4H^2+\kappa\geq 0$. If $\nu^2\in L^1(\Sigma)$, then $\kappa+4H^2=0$ and $\Sigma$ is a vertical cylinder over a complete curve in $\M^2(\kappa)$ with constant geodesic curvature $2H$.
\end{corollary}

\begin{proof}
Under these hypothesis,~\cite[Corollary~5.1]{Espinar} implies that $\kappa\leq 0$ and $\Sigma$ has critical mean curvature, and it is either a vertical cylinder or an entire graph. By the Daniel correspondence, we will suppose that $\E(\kappa,\tau)=\Nil(\tau)$ without loss of generality (stability is also preserved by the correspondence, see~\cite[Proposition~5.12]{Dan}). Nonetheles, if $\Sigma$ is an entire minimal graph in $\Nil(\tau)$ given by a function $u\in\mathcal{C}^\infty(\R^2)$, Theorem~\ref{last-area-estimate} gives the estimate
\[\int_\Sigma\nu^2=\int_{\R^2}\frac{1}{\sqrt{1+\|Gu\|^2}}\geq\int_0^\infty\frac{2\pi r}{\sqrt{1+B^2(1+r^2)^2}}=\infty.\]
Note that $B$ is a constant and we used polar coordinates to get to the last integral. Hence $\Sigma$ must be a vertical plane. As vertical surfaces are preserved by the correspondence, we are done.
\end{proof}

Estimates on the growth of the height of minimal graphs seem to be quite useful in the comprehension of the area growth. As far as we know there is no example whose height grows more than quadratically. In fact, we conjecture that the height of an entire minimal graph grows at most quadratically, for what it suffices to prove that the estimate in Lemma~\ref{lemma:gradient-estimate-L3} can be improved to $|\nabla v|\leq 1-A(1+r^2)^{-1}$.

In the next section, we will give a result in the opposite direction, by proving that the growth of the height of a minimal graph in $\Nil(\tau)$ with zero boundary values over an unbounded domain is at least linear.

\section{Height growth estimates \`{a} la Collin-Krust}

This section is devoted to study the behavior at infinity of the height of a minimal graph $\Sigma$ in $\Nil(\tau)$ with zero boundary values over an unbounded domain $\Omega\subset\R^2$. The very first results on this subject are due to Collin and Krust~\cite{CK} for minimal surfaces in $\R^3$. Some generalizations to the $\E(\kappa,\tau)$-setting have been obtained by Leandro and Rosenberg~\cite[Theorem 2]{LR}. We shall give a sharper result by comparing $\Sigma$ with the zero section, inspired by the linear height growth of catenoids (see Section~\ref{minimal-graph-section}).

\begin{theorem}\label{collin-krust-theo}
Let $\Omega\subset\R^2$ be an unbounded domain and let $u\in\mathcal{C}^\infty(\Omega)$ be a non-constant function whose graph over $\Omega$ is minimal with respect to the $\Nil(\tau)$-metric, such that $u$ extends continuously as zero to $\partial\Omega$. If we denote $M(r)=\sup_{\Omega\cap D_R(0)}|u|$, then
\[\liminf_{r\to\infty}\frac{M(r)}{r}>0.\]
If additionally there exists $C>0$ such that $\Length(\Omega\cap\partial D_R(0))<C$,  then 
\[\liminf_{r\to\infty}\frac{M(r)}{r^2}>0.\]
\end{theorem}

In order to prove Theorem \ref{collin-krust-theo} we need the following result that is proved in \cite{LR}. We will follow the notation of Section~\ref{minimal-graph-section}, and also denote by $N_u$ the upward-pointing normal vector field to the graph given by the function $u$.

\begin{lemma}\label{lemma:factorization}
For any $u,v\in C^1(M)$, 
\[\left\langle\frac{Gu}{W_u}-\frac{Gv}{W_v},Gu-Gv\right\rangle=\frac{1}{2}(W_u+W_v)\,|N_u-N_v|^2\geq 0.\]
Equality holds at some point $p\in M$ if and only if $\nabla u(p)=\nabla v(p)$.
\end{lemma}

\begin{proof}[Proof of Theorem \ref{collin-krust-theo}]
We will assume that $A=\{x\in\Omega:u(x)>0\}$ is not bounded and connected (we can restrict to a connected component if necessary and changing the sign of $u$ does not affect the arguments below). Given $r>0$, we define the sets $A(r)=A\cap D_r(0)$ and $\Lambda(r)=A\cap \partial D_r(0)\subset\partial A(r)$.  
Notice that  $W_0=\sqrt{1+\tau^2\rho^2}$, being $\rho$ the distance to the origin in $\R^2$. Moreover
the fact  that $u$ is positive on $A$ implies that there exists $r_0>0$ satisfying
$\mu=\int_{A(r_0)}|\frac{Gu}{W_u}-\frac{Z}{W_0}|^2>0$.

Let us define $\eta(r)=\int_{\Lambda(r)}|\frac{Gu}{W_u}-\frac{Z}{W_0}|$, for all $r\geq r_0$. Using Lemma~\ref{lemma:factorization}, the divergence theorem, the conditions $u=0$ along $\partial\Omega$ and $H(u)=0$ in $\Omega$, and the fact that $|N_u-N_0|\geq|\frac{Gu}{W_u}-\frac{Z}{W_0}|$, we can estimate for all $r\geq r_0$
\[\begin{aligned}
M(r)\eta(r)&\geq\int_{\partial A(r)}u\left|\frac{Gu}{W_u}-\frac{Z}{W_0}\right|\geq\int_{\partial A(r)}u\left\langle\frac{Gu}{W_u}-\frac{Z}{W_0},\chi\right\rangle=\int_{A(r)}\div\left(u\left(\frac{Gu}{W_u}-\frac{Z}{W_0}\right)\right)\\&=\int_{A(r)}\left\langle Gu-Z,\frac{Gu}{W_u}-\frac{Z}{W_0}\right\rangle=\int_{A(r)}\frac{W_u+W_0}{2}\left|\frac{Gu}{W_u}-\frac{Z}{W_0}\right|^2,
\end{aligned}\]
where $\chi$ denotes a unit conormal vector field to $A(r)$ along its boundary. We decompose the last integral in two integrals, one over $A(r_0)$, where we estimate $W_u\geq 1$ and $W_0\geq 1$, and another one over $A(r)\setminus A(r_0)$, where we estimate $W_u\geq 1$ and $W_0\geq\tau\rho-1$. We obtain
\begin{align}
M(r)\eta(r)&\geq\int_{A(r_0)}\left|\frac{Gu}{W_u}-\frac{Z}{W_0}\right|^2+\int_{A(r)\setminus A(r_0)}\frac{\tau\rho}{2}\left|\frac{Gu}{W_u}-\frac{Z}{W_0}\right|^2\notag\\
&=\mu+\frac{\tau}{2}\int_{r_0}^r s\left(\int_{\Lambda(s)}\left|\frac{Gu}{W_u}-\frac{Z}{W_0}\right|^2\right)\df s\geq \mu+\frac{\tau}{2}\int_{r_0}^r\frac{s\,\eta(s)^2\df s}{\Length(\Lambda(s))}.\label{eqn:estimate2}
\end{align}
As $\Length(\Lambda_s)\leq 2\pi s$, we conclude  that there exists $m>0$ such that, for all $r\geq r_0$,
\[M(r)\eta(r)\geq\mu+m\int_{r_0}^r \eta(s)^2\df s.\]

The function $r\mapsto M(r)$ is non decreasing by definition. Given $r_1>r_0$, let us write $a=M(r_1)$ so $a\,\eta(r)\geq M(r)\eta(r)$ for all $r_0<r<r_1$. Hence $\eta$ satisfies the integral inequality $\eta(r)\geq\frac{\mu}{a}+\frac{m}{a}\int_{r_0}^r\eta(s)^2\df s$. Let us define the function $\zeta:[r_0,L)\to\R$ as 
\[\zeta(r)=\frac{a\mu}{{2a^2}-m\mu(r-r_0)}, \qquad L=r_0+\tfrac{2a^2}{m\mu},\]
and observe that $\zeta(r)=\frac{\mu}{2a}+\frac{m}{a}\int_{r_0}^r\zeta(s)^2\df s$, so a simple comparison yields $\eta\geq\zeta$ for all $r_0\leq r\leq L$. Since $\eta$ is well-defined for all $r\geq r_0$ and $\zeta$ diverges when $r\to L$, we conclude that $r_1\leq L=r_0+\frac{2a^2}{m\mu}$. 
Equivalently,
\begin{equation}
\label{eqn:estimate3}M(r_1)=a\geq \sqrt{\frac{m\mu}{2}(r_1-r_0)},\qquad\text{for all }r_1>r_0.\end{equation}

We claim that the function $\eta$ is bounded away from zero at infinity. Note that, for any $r>r_0$,
\begin{equation}\label{eqn:estimate4}
\eta(r)\geq \left|\int_{\Lambda(r)}\left\langle
\frac{Gu}{W_u}-\frac{Z}{W_0},\chi\right\rangle\right|=\left|\int_{\partial A(r)}\left\langle
\frac{Gu}{W_u}-\frac{Z}{W_0},\chi\right\rangle-\int_{\partial A(r)\setminus \Lambda(r)}\left\langle
\frac{Gu}{W_u}-\frac{Z}{W_0},\chi\right\rangle\right|
\end{equation}
The first integral of the RHS of \eqref{eqn:estimate4} vanishes by Stokes Theorem. As for the second integral, we proceed as follows. We prove that $\int_{\Gamma}\langle\frac{Gu}{W_u}-\frac{Z}{W_0},\chi\rangle$ has constant sign on any arc $\Gamma$ contained in $\partial A$ (different from one point). Notice that, $Gu-Z=\nabla u\not=0$ along $\partial A$, except at isolated points, because $u\geq 0$ in $A$ by assumption. In particular, $Gu-Z$ is oriented towards $A$, where it is not zero. Hence $Gu-Z$ can be used to orient $\partial A$. Then, if $\langle\frac{Gu}{W_u}-\frac{Z}{W_0}, Gu-Z\rangle$ has constant sign along $\partial A$, the same holds for $\langle\frac{Gu}{W_u}-\frac{Z}{W_0},\chi\rangle$. By Lemma \ref{lemma:factorization},
\[\left\langle\frac{Gu}{W_u}-\frac{Z}{W_0},Gu-Z\right\rangle=\frac{1}{2}(W+W_0)|N_u-N_0|^2\]
is positive at any point where $Gu-Z$ is not zero. Then there exists a constant $c$ such that $\eta(r)\geq\int_{\Gamma}\langle\frac{Gu}{W_u}-\frac{Z}{W_0},\chi\rangle\geq c>0$, which proves the claim.

For any $r_2>r_0$ we deduce that
\begin{equation}\label{eqn:estimate5}
\mu(r_2):=\int_{A(r_2)}\left|\frac{Gu}{W_u}-\frac{Z}{W_0}\right|^2\geq\frac{\tau}{2} \int_{r_0}^{r_2}\frac{s\,\eta(s)^2}{\Length(\Lambda(s))}ds\geq\frac{\tau c^2}{2}\int_{r_0}^{r_2}\frac{s}{\Length(\Lambda(s))}ds\geq\frac{\tau c^2}{4\pi}(r_2-r_0),
\end{equation}
where the first inequality follows from~\eqref{eqn:estimate2}, the second one from the claim above, and the third one from the fact that ${\Length(\Lambda(s))}\leq 2\pi s$. Applying~\eqref{eqn:estimate3} to $r_2=\frac{r_0+r_1}{2}<r_1$ instead of $r_0$ we get
\begin{equation}\label{eqn:estimate6}
M(r_1)\geq\sqrt{\frac{m\mu(r_2)}{2}(r_1-r_2)}\geq\frac{c\sqrt{m\tau}}{2\sqrt{\pi}}\sqrt{(r_1-r_2)(r_2-r_0)}=\frac{c\sqrt{m\tau}}{4\sqrt{\pi}}(r_1-r_0)
\end{equation}
for all $r_1>r_0$, which gives the desired estimate.

Assume now that there exists $C>0$ such that $\Length(\Lambda(r))\leq C$. Proceeding as above, inequality~\eqref{eqn:estimate2} now reads 
\[M(r)\eta(r)\geq\mu+n\int_{r_0}^rs\eta(s)^2\df s\]
for some constant $n>0$. Let us take $r_3>r_0$ and write $b=M(r_3)$ so $b\,\eta(r)\geq M(r)\eta(r)$ for all $r_0<r<r_3$. Hence, $\eta$ satisfies the integral inequality $\eta(r)\geq\frac{\mu}{b}+\frac{n}{b}\int_{r_0}^r s\eta(s)^2\df s$. Let also be
\[\xi:\left[r_0,\sqrt{\tfrac{4b^2}{n\mu}+r_0^2}\right[\to\R,\qquad\xi(r)=\frac{2b\mu}{{4b^2}-n\mu(r^2-r_0^2)},\]
which satisfies $\xi(r)=\frac{\mu}{2b}+\frac{n}{b}\int_{r_0}^r s\,\xi(s)^2\df s$, so $\eta\geq\xi$ for all $r_0\leq r\leq (\tfrac{4b^2}{n\mu}+r_0)^{1/2}$ by comparison. Since $\xi$ diverges when $r\to (\tfrac{4b^2}{n\mu}+r_0)^{1/2}$, we conclude that $(\tfrac{4b^2}{n\mu}+r_0)^{1/2}\geq r_3$. Equivalently,
\begin{equation}\label{eqn:estimate7}M(r_3)=b\geq \sqrt{\frac{n\mu}{4}(r_3^2-r_0^2)}\qquad \text{for all }r_3>r_0.
\end{equation}
In this case, instead of inequality \eqref{eqn:estimate5}, we have $\mu(r_4)\geq\frac{\tau c^2}{4C}(r_4^2-r_0^2)$ for all $r_4>r_0$. Taking $r_4=(\frac{r_3^2+r_0^2}{2})^{1/2}$ rather than $r_0$ in~\eqref{eqn:estimate7}, it becomes
\begin{equation}\label{eqn:estimate9}
M(r_3)\geq \frac{c\sqrt{\tau n}}{2\sqrt{C}}\sqrt{(r_3^2-r_4^2)(r_4^2-r_0^2)}=\frac{c\sqrt{\tau n}}{4\sqrt{C}}(r_3^2-r_0^2)
\end{equation}
for all $r_3>r_0$, which finishes the proof.
\end{proof}

\begin{remark} Theorem \ref{collin-krust-theo} holds in a more general case. Our hypothesis that the function $u$ has zero boundary value enables us to compare $u$ with the zero section, and the key property for our improvement is that $W_0$ is a radial function that grows linearly, which gives a sharper result than simply taking $W_0\geq 1$. In fact, this technique can be adapted to Killing submersions having a rotational symmetry.

We conjecture that, given $u,v\in\mathcal{C}^\infty(\Omega)$ spanning minimal graphs and such that $u=v$ along $\partial\Omega$, the same result as in Theorem~\ref{collin-krust-theo} holds for $M(r)=\sup_{\Omega\cap D_R(0)}|u-v|$, provided that $u-v$ is not constant. Nevertheless, it does not seem that the proof of Theorem~\ref{collin-krust-theo} or the arguments in~\cite{LR} can be easily adapted to this more general situation.
\end{remark}

As a consequence of Theorem \ref{collin-krust-theo}, we generalize the fact that a bounded minimal graph with zero boundary values is unique (see \cite{LR}).

\begin{corollary}\label{coro:sublineal-height}
Let $\Omega\subset\R^2$ be an unbounded domain.
\begin{itemize}
 	\item[(a)]  If $\partial\Omega\neq \emptyset$, then the only minimal graph over $\Omega$ in $\Nil(\tau)$ with zero boundary values and sublinear growth (i.e., such that $\limsup_{r\to\infty} \frac{M(r)}{r}=0$) is given by the constant zero $u\equiv 0$.
 	\item[(b)] In the case $\partial\Omega=\emptyset$, it follows that the only entire minimal graphs  in $\Nil(\tau)$ with sublinear growth are the constant ones.
 \end{itemize}
\end{corollary}

Let us make a final remark about the height of a minimal graph in $\Nil(\tau)$. In Theorem~\ref{last-area-estimate}, we prove that the height of an entire minimal graph  in $\Nil(\tau)$ grows at most cubically, and Theorem~\ref{collin-krust-theo} shows that it is at least linear (unless the graph is constant). This result is sharp, as half of a catenoid or planes of the form $u(x,y)=ax+by$ show. Other non-trivial examples of graphs over a sector with angle between $\frac{\pi}{2}$ and $\pi,$ with either linear or at least quadratic height growth, are given in~\cite{C} and~\cite{NST} (see also Section~\ref{minimal-graph-section}).

\end{document}